\documentclass[11pt]{article}

\usepackage[margin=1.08in]{geometry}
\usepackage{amsmath,amssymb,amsthm,mathtools,mathrsfs}
\usepackage{enumitem}
\usepackage{xcolor}
\usepackage{hyperref}
\usepackage[nameinlink,capitalize]{cleveref}
\usepackage{microtype}

\hypersetup{
  colorlinks=true,
  linkcolor=blue!60!black,
  citecolor=blue!60!black,
  urlcolor=blue!60!black
}

\newtheorem{theorem}{Theorem}[section]
\newtheorem{proposition}[theorem]{Proposition}
\newtheorem{lemma}[theorem]{Lemma}

\theoremstyle{definition}
\newtheorem{definition}[theorem]{Definition}
\newtheorem{assumption}[theorem]{Assumption}
\theoremstyle{remark}

\numberwithin{equation}{section}
\numberwithin{theorem}{section}

\newcommand{\R}{\mathbb R}

\newcommand{\E}{\mathbb E}
\newcommand{\Prob}{\mathbb P}
\newcommand{\Var}{\operatorname{Var}}

\newcommand{\Tr}{\operatorname{Tr}}
\newcommand{\rank}{\operatorname{rank}}
\newcommand{\Span}{\operatorname{span}}

\newcommand{\op}{\mathrm{op}}
\newcommand{\diag}{\operatorname{diag}}
\newcommand{\Law}{\operatorname{Law}}
\newcommand{\dist}{\operatorname{dist}}
\newcommand{\ip}[2]{\left\langle #1,#2\right\rangle}
\newcommand{\norm}[1]{\left\lVert #1\right\rVert}
\newcommand{\abs}[1]{\left\lvert #1\right\rvert}
\newcommand{\taud}{\tau_d}
\newcommand{\taup}{\tau}
\newcommand{\cA}{\mathcal A}

\newcommand{\ons}{\operatorname{ons}}
\newcommand{\Av}{\operatorname{Av}}

\newcommand{\xto}{\xrightarrow}
\newcommand{\ind}{\mathbf 1}
\newcommand{\Symm}{\mathbb{R}^{d \times d}_{\mathrm{sym}}}
\newcommand{\PL}{\operatorname{PL}}
\newcommand{\eps}{\varepsilon}
\newcommand{\rmd}{\mathrm{d}}
\newcommand{\one}{\mathbf{1}}

\title{Free-Probabilistic State Evolution and Random Matrix Discrepancy}
\author{
August Y. Chen \thanks{Cornell University, Department of Computer Science. Ithaca, NY, USA. Email: \texttt{ayc74@cornell.edu}}
\and
Ahmed El Alaoui \thanks{Cornell University, Department of Statistics and Data Science. Ithaca, NY, USA. Email: \texttt{elalaoui@cornell.edu}.}
}
\date{}

\begin{document}
\maketitle

\begin{abstract}
Let $A_1,\ldots,A_n$ be independent $d \times d$ real symmetric Gaussian random matrices, and consider the linear operator $\cA(x) = n^{-1/2}\sum_{i=1}^n x_iA_i$, $x\in \R^n$. We construct an iterative algorithm in the Approximate Message Passing family which iterates over $\cA$ and its adjoint $\cA^*$, and establish a state evolution result which  characterizes its behavior in the limit $d\to \infty, 2n/d^2 \to \alpha$ in terms of a correlated Gaussian-semicircular process in a free probability space, in the sense of strong convergence of operators.  We then apply this iteration to the random matrix discrepancy problem which asks for a binary vector $x \in \{-1,+1\}^n$ such that $\cA(x)$ has a small operator norm. Our algorithm achieves an operator norm $2\sigma(\alpha)$, for a standard deviation $\sigma(\alpha)<1$ for all $\alpha>0$, and $\sigma(\alpha) = O(1/\sqrt{\alpha \log \alpha})$ for large $\alpha$.  This resolves the algorithmic question of~\cite{kunisky2023average,maillard2025average}.      
\end{abstract}

\tableofcontents

\section{Introduction}
In this paper we study an Approximate Message Passing (AMP) iteration applied to a random, matrix-valued linear operator $\cA$ defined as follows:
let $A_1,\ldots,A_n$ be symmetric $d \times d$ independent random matrices with independent Gaussian upper-triangular entries:
\begin{equation}\label{eq:GOE}
  (A_i)_{jk} = (A_i)_{kj} \sim N(0,1/d)\quad (j<k),
  \qquad
  (A_i)_{jj}\sim N(0,2/d)\,,
\end{equation}
and define the random operator $\cA : \R^n \to \Symm$, where $\Symm$ denotes the real vector space of symmetric $d\times d$ matrices, by
\begin{equation}
    \cA (x) = \frac{1}{\sqrt{n}} \sum_{i=1}^n x_i A_i\,.
\end{equation}
 This is motivated by a random version of the matrix discrepancy problem due to~\cite{kunisky2023average}, which asks whether it is possible to produce a binary vector $x \in \{-1,+1\}^n$ such that the operator norm $\|\cA(x)\|_{\op}$ is strictly smaller than 2 in the large dimension limit. (By Wigner's semicircle law, $\|\cA(x)\|_{\op}=2+o_d(1)$ if $x/\sqrt{n}$ is a unit vector chosen independently of the matrices $A_i$.) In Section~\ref{sec:matrix-discrepancy} we propose an algorithmic solution to this question achieving an explicit expression for the operator norm as a function of the ratio  
\[\alpha_d =  \frac{2n}{d^2} \,.\]
We now set up notation that will be used throughout this paper. For $M, N \in \Symm$ we define
\[
  \taud(M):=\frac1d\Tr(M)\,,
  \qquad
  \ip{M}{N}_d:=\taud(MN)\,,
  \qquad
  \norm{M}_{2,d}^2:=\taud(M^2)\,,
\]
and for $x,y \in \R^n$,
\[
  \ip{x}{y}_n:=\frac1n\sum_{i=1}^n x_i y_i\,,
  \qquad
  \norm{x}_n^2:=\frac1n\sum_{i=1}^n x_i^2\,.
\]
Let
$\cA^* : (\Symm, \ip{\cdot}{\cdot}_d) \to (\R^n, \ip{\cdot}{\cdot}_n)$ denote the adjoint of $\cA$ with respect to the indicated normalized inner products. More explicitly,
\begin{equation}
    \cA^* (M) = \sqrt{n} \big(\ip{A_i}{M}_d\big)_{i=1}^n\,.
\end{equation}
We consider the following family of vector--matrix AMP iterations applied to $\cA$ and $\cA^*$: 
\begin{align}
  u^{t+1}&=\cA^*(M^t)- \sum_{r=0}^{t} a_{t,r} m^r\,,
  \qquad
  M^t=\sum_{s=0}^{t} a_{t,s} V^{s}\,,\label{eq:AMP0}\\
  V^t &=\cA(m^t)-\alpha_d \sum_{r=0}^t d_{t,r} M^{r-1}\,,
  \qquad
  m_i^t=g_t(u_i^0,\ldots,u_i^t)\,,\label{eq:AMP1}
\end{align}
where the scalar coefficients $a_{t,s}$ and the maps $g_t : \R^{t+1} \to \R$ are fixed, and 
\begin{equation}\label{eq:d_{rs}}
d_{t,r} = \frac{1}{n}\sum_{i=1}^n \partial_{r}g_{t}(u_i^0,\ldots,u_i^{t})\,.
\end{equation} 

The iteration is initialized at time $t=0$ with $u^0\in \R^n$ and run up to a finite time horizon $T$. We use the convention $M^{-1} = V^{-1} = 0$.
The subtracted terms in the above iteration are called the \emph{Onsager correction} terms and are denoted as
\begin{align}\label{eq:onsdef}
  \ons_u^t :=\sum_{r=0}^{t} a_{t,r} m^r\,,
  \qquad
\ons_V^t := \alpha_d \sum_{r=0}^t d_{t,r} M^{r-1}\,.
\end{align}

We wish to establish an exact asymptotic description of all vector and matrix iterates through every fixed time $t$, in the high-dimensional limit $d\to\infty$ with
\begin{equation}\label{eq:alpha}
  \alpha_d =  \frac{2n}{d^2} \longrightarrow \alpha\in(0,\infty)\,.
\end{equation}  
The result will take the form of a \emph{state evolution} statement which describes the \emph{empirical} convergence of entries on the `$u$-side' and the \emph{strong} convergence of operators on the `$V$-side' in the sense of free probability. The limit is described by a correlated Gaussian-semicircular process. 

\section{Limiting state evolution}
We start by making explicit assumptions on the iteration's data. 
We say that a function $\psi:\R^k\to\R$ is pseudo-Lipschitz of order $m$, and write $\psi \in \PL_m(\R^k)$ if there exist $L>0$ such that for all $x,y \in \R^k$,
\begin{equation}
    \big|\psi(x) - \psi(y)\big| \le L\big(1+\|x\|_2^{m-1}+\|y\|_2^{m-1}\big) \|x-y\|_2\,.
\end{equation}

\begin{assumption}[Aspect ratio and matrix normalization]\label{ass:aspect}
The matrices $A_1,\ldots,A_n$ are independent GOE matrices with the normalization of Eq.~\eqref{eq:GOE}, and $\alpha_d=2n/d^2\to\alpha\in(0,\infty)$.
\end{assumption}

\begin{assumption}[Initialization]\label{ass:init}
The triangular array $(u_d^0)_{d\ge1}$ is defined on a common probability space, is independent of the Gaussian operators, and there is a random variable $U^0$, with moments of all orders, such that the empirical
measures
\[
  \widehat\mu_{u_d^0}:=\frac1n\sum_{i=1}^n\delta_{u_{d,i}^0}
\]
converge almost surely to $\Law(U^0)$ in $p$-Wasserstein distance for every finite $p$. Equivalently, almost surely, the empirical measures converge weakly and all empirical moments converge to the corresponding moments of $U^0$.
\end{assumption}

\begin{assumption}[Scalar coefficients]\label{ass:a1}
The coefficients $(a_{t,s})$ are deterministic, and $0 <a_* \le \min_{t \le T} |a_{t,t}|$ and  $ \max_{s \le t\le T} |a_{t,s}| < a^*$ for some positive finite pair $a_*<a^*$. 
\end{assumption}

\begin{assumption}[Nonlinearities]\label{ass:g}
For each $0\le t\le T$, the function
$g_t:\R^{t+1}\to\R$ is $C^1$.  There is a finite integer
$k_t$ such that
\[
 g_t,\partial_0g_t,\ldots,\partial_tg_t
 \in \PL_{k_t}(\R^{t+1})\,.
\]
\end{assumption}

\begin{assumption}[Linearity in last iterate]
\label{ass:transversality}
The initial map satisfies $g_0(x)\ne 0$ for every $x\in\R$.  In
particular, $m^0=g_0(u^0)\ne0$ and $G^0=g_0(U^0)\ne0$ almost surely.
For every $1\le t\le T$, there are measurable functions
$r_t,c_t:\R^t\to\R$ such that
\begin{equation}\label{eq:transverse-form}
 g_t(x^0,\ldots,x^t)
 =r_t(x^0,\ldots,x^{t-1})
  +c_t(x^0,\ldots,x^{t-1})x^t\,,
 \qquad c_t>0\,.
\end{equation}
\end{assumption}

We now describe the putative limiting process.
Let $U^0$ be the limiting initialization as per Assumption~\ref{ass:init}.  
Let the random variables $(U^1,\ldots,U^{T+1})$ be centered jointly Gaussian and independent of $U^0$ defined on a common probability space $(\Omega,\mathcal{F},\Prob)$,  and $(S^0,\ldots,S^T)$ be a centered semicircular family in a $C^*$-probability space $(\mathcal N,\tau)$. The trace $\tau$ is \emph{faithful}, meaning that $\tau(a^*a)=0$ implies $a=0$. This allows to define the operator norm $\|\cdot\|$ on $\mathcal N$ by the formula
\begin{equation}\label{eq:norm_N}
\|a\| = \lim_{p \to \infty} \tau((a^*a)^p)^{\frac{1}{2p}}\,.
\end{equation}
We refer the reader to~\cite{mingo2017free} for an introduction to free probability and relevant background.  
The covariances of $(U^1,\cdots,U^{T+1})$ and of $(S^0,\cdots,S^T)$ are defined recursively as follows: for $0\le r,s\le T$,
\begin{align}
\E[U^{r+1}U^{s+1}] &= \alpha\,\taup(W^rW^s)\,,
  \qquad W^t:=\sum_{s=0}^{t} a_{t,s}S^s,\label{eq:covU}\\
  \taup(S^rS^s) &= \E[G^rG^s]\,,
  \qquad G^t:=g_t(U^0,\ldots,U^t)\,.\label{eq:covS}
\end{align}

Now we define the relevant notions of convergence.

\begin{definition}[Pseudo-Lipshitz empirical convergence]\label{def:emp-conv}
For random vectors $X_i^d\in\R^k$, $1\le i\le n = n(d)$, with $\lim_{d \to \infty} n(d) =\infty$, and a random vector $X\in\R^k$, we say that $(X_1^d,\ldots,X_n^d)$ converges empirically to $X \in \R^k$ as $d \to \infty$, and  
write
\[
  (X^d_i)_{i\le n}\,\xrightarrow{\mathrm{emp}}\, X\,,
\]
 if for every $m$ and every  $\psi\in \PL_m(\R^k)$,
\[
\frac1n\sum_{i=1}^n\psi(X_i^d)\xrightarrow{p}\E\psi(X)\,.
\]
\end{definition}

\begin{definition}[Strong convergence in noncommutative distribution]\label{def:strong}
Let $(X_1^d,\ldots,X_k^d)$ be a $k$-tuple of self-adjoint random matrices and let $(x_1,\ldots,x_k)$ be self-adjoint elements of the $C^*$-probability space $(\mathcal N,\taup)$.  We say that $(X_1^d,\ldots,X_k^d)$ converges strongly in noncommutative distribution to $(x_1,\ldots,x_k)$ as $d \to \infty$, and write
\[
  (X_1^d,\ldots,X_k^d) \xrightarrow{\mathrm{str}} (x_1,\ldots,x_k)
\]
if for every noncommutative polynomial $P$,
\[
  \taud\bigl(P(X_1^d,\ldots,X_k^d)\bigr)\xto{p}
  \taup\bigl(P(x_1,\ldots,x_k)\bigr)
\]
and
\[
  \norm{P(X_1^d,\ldots,X_k^d)}_{\op}\xto{p}
  \norm{P(x_1,\ldots,x_k)}\,,
\]
 where $\|\,\cdot\,\|_{\op}$ is the matrix operator norm and $\|\,\cdot\,\|$ is the operator norm on $\mathcal{N}$ defined in~\eqref{eq:norm_N}. 
\end{definition}

We now state the state evolution theorem. Fix an integer $T\ge1$, run the AMP iteration \eqref{eq:AMP0}--\eqref{eq:AMP1} up to time $T$, and let
$h:\R^{T+1}\to\R$ be bounded Borel function.  Define a terminal
message $\widehat{m}$ (not fed back into the iteration) by
\begin{equation}\label{eq:terminal-comparison-readouts}
 \widehat m_i:=h(u_i^0,\ldots,u_i^T)\,,
 \qquad\qquad
 H:=h(U^0,\ldots,U^T)\,.
 \end{equation}
and recall
\[  m_i^T:=g_T(u_i^0,\ldots,u_i^T)\,, \qquad G^T = g_T(U^0,\ldots,U^T)\,.\]

\begin{assumption}[Null discontinuity]
\label{ass:zero-discontinuity}
Let $\operatorname{Disc}(f)$ denote the set of discontinuity points of a function $f$ and assume
\begin{equation}\label{eq:terminal-comparison-null}
 \Prob\left(  (U^0,\ldots,U^T) \in\operatorname{Disc}(h) \right)=0\,.
\end{equation}
\end{assumption}

\begin{theorem}
\label{thm:main-general}
The following holds under Assumptions \ref{ass:aspect}, \ref{ass:init},
\ref{ass:a1}, \ref{ass:g}, \ref{ass:transversality} and \ref{ass:zero-discontinuity}.  For fixed $T \ge 1$, let $(u^{t},m^{t},V^{t},M^{t})_{t \le T}$ be generated by
the AMP iteration \eqref{eq:AMP0}--\eqref{eq:AMP1}.  Then, 
\begin{align}
  \bigl(u_i^0,\ldots,u_i^{T},m_i^0,\ldots,m_i^{T}, \widehat m_i\bigr)_{i\le n}
  &\xrightarrow{\mathrm{emp}}
  \bigl(U^0,\ldots,U^{T},G^0,\ldots,G^T,H\bigr)\,,\label{eq:vector-SE-main}\\
\mbox{and}\qquad
  (V^0,\ldots,V^T,M^0,\ldots,M^T)
  &\xrightarrow{\mathrm{str}}
  (S^0,\ldots,S^T,W^0,\ldots,W^T)\,,\label{eq:matrix-SE-main}
\end{align}
where the limiting process
$(U^0,\ldots,U^{T},S^0,\ldots,S^T)$ has covariances given by
\eqref{eq:covU}--\eqref{eq:covS}.
Moreover, for every $\eta>0$,
\begin{equation}\label{eq:terminal-comparison-opnorm}
 \Prob\left(
 \bigl\|\cA(\widehat m-m^T)\bigr\|_{\op}
 >
 2(1+\sqrt\alpha)\,
 \E[(H-G^T)^2]^{1/2}
 +\eta
 \right)
 \longrightarrow 0\,.
\end{equation}
\end{theorem}

The study of AMP algorithms and their limiting behavior is well established and has been instrumental in many advances in high-dimensional probability and statistics. The available results in the literature concern iterates involving a random operator $A : \R^n \to \R^m$ with vector inputs and outputs, and the state evolution results are concerned with establishing empirical convergence on both sides; see for instance~\cite{javanmard2013hypothesis,bolthausen2014iterative,berthier2017state}. In our case we are interested in tracking the spectral content of the matrix iterates; empirical convergence of their entries is not strong enough to preserve spectral information. The main novelty of our result is that spectral information can tracked in terms of a sequence of semicircular operators in a free probability space, in the strong convergence sense~\cite{voiculescu1991limit,haagerup2005new,mingo2017free,van2025strong}. In particular in addition to asking for convergence of normalized traces of arbitrary polynomials in the iterates (i.e., weak convergence~\cite{voiculescu1991limit,mingo2017free}), the convergence of operator norms ensures that the full spectrum is accounted for in the limit: no outliers eigenvalues are missed. We refer to~\cite{van2025strong} for an expository article on strong convergence. 
This opens the possibility of designing a rich variety of spectra by appropriately choosing the parameters of the AMP iteration.     
            
 We also comment on the presence of the terminal readout $\widehat m$~\eqref{eq:terminal-comparison-readouts} and its image by $\cA$ in the statement. The aforementioned classical state evolution results operating on a random data matrix $A: \R^n \to \R^m$ are stated without this terminal step since it is usually redundant: the matrix $A$ has bounded operator norm. In our case it can be verified that the norm $\|\cA\|_{(\R^n,\|\cdot\|_n)\to(\Symm,\|\cdot\|_{\op})}$ is of order $\sqrt d$. The significance of this is that rounding the last AMP iteration $m^T$ becomes a nontrivial issue. Indeed, in the matrix discrepancy application we consider below, the operator norm stability estimate~\eqref{eq:terminal-comparison-opnorm} (which does not follow from crude considerations) will allow us to round the output $m^T$ into a binary vector $\widehat m$ obtained by taking the entrywise sign: $h = \mathrm{sign} \circ g_T$ without deteriorating the operator norm guarantee. 
 
 We also mention that a stronger result which appends $\cA(\widehat m)$ to the joint strong convergence statement~\eqref{eq:matrix-SE-main} is possible: the limit is jointly semicircular with an explicit covariance but this will not be needed for our purposes.

\section{Application to random matrix discrepancy} 
\label{sec:matrix-discrepancy}
We now apply the above result to the matrix discrepancy problem.
Zouzias~\cite{zouzias2012matrix} and Meka~\cite{meka2014blog} asked whether there exists a finite constant $C$ such that for any sequence  $A_1,\ldots,A_n$ of $d \times d$ symmetric matrices with $\max_{i\in [n]} \|A_i\|_{\op}\le 1$, one has
\[\min_{x \in \{-1,+1\}^n}\Big\|\sum_{i=1}^n x_i A_i\Big\|_{\op} \le C \sqrt{n\max\bigl(1,\log(d/n)\bigr)}\,.\]
This is the matrix version of the more standard vector discrepancy problem where the matrices $A_i$ are replaced by vectors $v_i$ and the relevant norm is $\ell_\infty$; a special case of the above by considering diagonal matrices. The vector case is well understood and is the foundation of discrepancy minimization theory, starting with Spencer's celebrated existence result~\cite{spencer1985six}, with several efficient algorithmic constructions proposed; see for instance~\cite{bansal2013deterministic,lovett2015constructive,levy2017deterministic}. 
The matrix discrepancy problem has seen a surge of recent activity with partial results~\cite{dadush2022new,hopkins2022matrix,bansal2023resolving,bandeira2026matrix,akbas2026algebraic}, culminating in the very recent proposal of a full solution~\cite{akbas2026algebraic}, and several subsequent extensions  \cite{akbas2026boolean, song2026matrix, bansal2026simpler, kathuria2026walk}. 

The \emph{random} version of this problem, proposed by~\cite{kunisky2023average}, considers independent random matrices $A_1,\ldots,A_n$. The authors provide upper and lower estimates for the discrepancy as a function of the parameters $n,d$ and analyze an online algorithm yielding a discrepancy of $d 
\log n$ for a general class of random matrices; see also~\cite{kunisky2024asymptotic}. In the case of Wigner matrices, one can ask the more refined question of whether there exists a signing which compresses the semicircle law below its typical spectral radius 2, with high probability. Maillard~\cite{maillard2025average} proved the following result:
\begin{theorem}[\cite{maillard2025average}]\label{thm:maillard}
    Let $\sigma<1$ and $2n/d^2 \to \alpha$. There exists $0 < \alpha_1< \alpha_2 <\infty$ such that
    \[\min_{x \in \{-1,+1\}^n} \big\|\cA(x)\big\|_{\op} > 2\sigma \,,  \]
    with high probability if $\alpha<\alpha_1$, and  
    \[\min_{x \in \{-1,+1\}^n} \big\|\cA(x)\big\|_{\op} \le 2\sigma \,,  \]
    with high probability if $\alpha>\alpha_2$. 
    Moreover, $\alpha_1, \alpha_2 \to +\infty$ as $\sigma \to 0$, and $\alpha_2(\sigma = 1^-) \ge 11$.
\end{theorem}
His proof relies on the first and second moment methods, where one needs to control large deviation events on the spectra of correlated Gaussian random matrices. He asked whether the positive result can be achieved algorithmically (\cite[Open Question 1.4]{maillard2025average}). Our result provides a positive answer. We write a general feasibility theorem which we later delineate in two ways.   

Shortly after the initial posting of this paper, independent and concurrent work by K{\i}z{\i}lda{\u{g}} and Li \cite{kizildaug2026stable} presented an efficient algorithm to find a signing $x \in \{-1,+1\}^n$ achieving operator norm $\|\cA(x)\|_{\op} < 2$ in the regime of large $\alpha$. In particular, their algorithm achieves operator norm $O(1/\sqrt{\alpha \log \alpha})$ for $\alpha$ large, matching the guarantees of Theorem \ref{thm:large-alpha}, the consequence of our general feasibility theorem in this regime. In the same regime of $\alpha$ large, they also provided a matching lower bound for stable algorithms based on the overlap gap property. 

\paragraph{Our results.}
Let $(B_t)_{t \ge0}$ be a standard Brownian motion on a probability space $(\Omega, \mathcal F, \Prob)$. Let $\mathbb F = (\mathcal F_t)_{t\ge0}$ be its (completed) natural filtration and $(b_t)_{t\ge0}$ be a $\mathbb F$-predictable process (e.g., $\mathbb F$-adapted and left-continuous), which we call a \emph{control}, and let $T\ge 0$ be such that 
\begin{equation}
\label{eq:positivity_b}
b_s \ge 0 ~~~\mbox{a.s.}\,, \qquad \int_0^T \E b_s^2 \,\rmd s < \infty\,.
\end{equation}
 Consider the martingale 
\begin{equation}
\label{eq:matringale_G_T}	
G_t = \int_0^t b_s \,\rmd B_s\,,\quad ~\mbox{with}~\quad G_T \in \{-1,+1\}~\mbox{a.s.}
\end{equation}
For a given $\alpha >0$ and $b$ as above, define            
\begin{align}
 J_\alpha(b)
 &:=
\int_0^T \Big(\sqrt{\E b_s^2} - \sqrt{\alpha}\E b_s\Big)^2 \rmd s\,.
 \label{eq:J-alpha}
\end{align}

\begin{theorem}
\label{thm:random_matrix_discrepancy}
Assume $2n/d^2\to\alpha\in(0,\infty)$ and let $b$ satisfy~\eqref{eq:positivity_b}--\eqref{eq:matringale_G_T}. For every 
$\varepsilon>0$, there exists an efficient algorithm which linearly queries $\cA$ and its adjoint $\cA^*$ and outputs a sign vector $\widehat x\in\{-1,+1\}^n$ such that
\begin{equation}\label{eq:binary-final-bound}
 \Prob\Big(
 \|\cA(\widehat x)\|_{\op}
 \le 2\sqrt{J_\alpha(b)}+\varepsilon
 \Big)\longrightarrow 1\,.
\end{equation}
The number of calls to $\cA$ and $\cA^*$ is bounded by a constant
$N=N(\alpha,\varepsilon)$ independent of $d$.
\end{theorem}

The algorithm relies on an incremental version of the AMP iteration studied in the previous section. This approach is similar to, albeit much simpler than, the one developed in the context of the random perceptron model~\cite{alaoui2022algorithmic,huang2026algorithmic}. Incremental AMP was originally developed in the context of optimizing spin-glass Hamiltonians~\cite{subag2018following,montanari2021optimization,ams2020}, and has been shown to find an approximate ground state under an analytic assumption on a certain variational formula describing the structure of the near optimizers, known as full replica-symmetry breaking; this condition was recently proved by OpenAI~\cite{OAI.fullrsb_sk}. In the present case we do not appeal to such machinery: we only aim for a non-trivial bound and do not attempt to achieve the minimum possible operator norm.

Next we implement a relatively simple control making $(G_t)_{t\le T=1}$ the expected sign of $B_1$ given the current history. This ensures integrality at the terminal time, but does not attempt to minimize the functional $J_{\alpha}$.  
  
Let $\varphi_{\sigma^2}$ denote the density function of $N(0,\sigma^2)$ and let 
\[b_t = 2\varphi_{1-t}(B_t)\,,\quad t\in[0,1).\]
Let $\Phi$ be the cumulative distribution function of $N(0,1)$. We will verify in Section~\ref{sec:implementations} that 
\[G_t = \E \big[\operatorname{sign}(B_1) \mid B_t\big] = 2\Phi\Big(\frac{B_t}{1-t}\Big)-1\,,~~~ G_1 = \operatorname{sign}(B_1)\in\{-1,1\}~\mbox{a.s.}\,,\]
and
\begin{align}
 J_{\alpha}(b)
 &=
 1-\frac{2\sqrt\alpha}{\pi}
 B\left(\frac12,\frac34\right)
 +\frac{2\alpha}{\pi}\,,
 \label{eq:sigma-one-endpoint}
\end{align}
where $B$ is the classical beta function with
\[
B\left(\frac12,\frac34\right) =  \int_0^1\sqrt{t}(1-t)^{-1/4}\,\rmd t = 2.39628\ldots
\]
With this choice, the bound~\eqref{eq:binary-final-bound} is non-trivial in an interval in the neighborhood of zero: 
 \[0<\alpha<\alpha_* := B\left(\frac12,\frac34\right)^2 = 5.74216\ldots\,.\]
This interval is disjoint from what is covered by the positive result in Theorem~\ref{thm:maillard}.

  Our second implementation concerns large $\alpha$.   
Consider the family of probability densities $(f_\eta)_{\eta\ge0}$, 
\begin{align}\label{eq:profile-intro}
 f_\eta(z)=\frac{c_\eta}{(\eta+1-z^2)^2}\,,\quad -1\le z\le1\,,
 \qquad
 Z_\eta\sim f_\eta,
\end{align}
and let
\begin{align}
 \label{eq:profile-moments-intro}
 v_\eta=\E Z_\eta^2\,,\qquad  h_\eta &=2\E\varphi\left(\Phi^{-1}\Big(\frac{1+Z_\eta}{2}\Big)\right)\,\,,\quad \varphi = \varphi_1\,.
 \end{align}

    \begin{theorem}\label{thm:large-alpha}
    For every $\alpha > \alpha_0 = 32/(3\pi^2)$, there exists $\eta = \eta_\alpha>0$ and a control $b$ satisfying~\eqref{eq:positivity_b}--\eqref{eq:matringale_G_T} such that 
    \[J_\alpha(b) \le 1-v_\eta + \alpha h_\eta^2\,.\]
    Moreover, $\eta$ can be selected so that $\eta_\alpha \to 0$ as $\alpha \to \infty$, and   
    \[1-v_\eta + \alpha h_\eta^2 = O\big(1/(\alpha \log \alpha)\big)\,.\]
    Consequently, there exists constants $C,\alpha_1>0$ such that for all $\alpha > \alpha_1$, one can efficiently obtain a signing $\widehat x \in \{-1,+1\}^n$ such that  
    \begin{equation}\label{eq:binary-final-bound2}
 \Prob\Big(
 \|\cA(\widehat x)\|_{\op}
 \ge \frac{C}{\sqrt{\alpha \log \alpha}}
 \Big)\longrightarrow 0\,.
\end{equation}
    \end{theorem}

\subsection{Informal description of the control for large $\alpha$}
The idea behind the above result is to design a three-stage control $(b_t)_{0\le t\le T}$ implementing different strategies on the time intervals $[0,t_0], [t_0,1]$ and $[1,T]$, $T=2$. We start by describing the second, main stage. In the time interval $[t_0,1]$, we want to force the integrand in the variance expression~\eqref{eq:J-alpha} to vanish: 
  \begin{equation}\label{eq:moment-matching-b}  
  	\E b_t^2 = \alpha (\E b_t)^2\,,\qquad t \in [t_0,1],
  \end{equation}
while making sure that $|G_t|\le 1$ and $\E G_1^2$ is as large as possible. This allows to make progress in finding a signing with no variance accumulation.  The first stage $[0,t_0]$ is set up to properly initialize the second stage.   And in the third stage $t \in [1,2]$ we switch to the previous control. This forces sign integrality at the terminal time $t=2$ while ensuring that the variance has very little time to accumulate. 
  
A solution to the second stage is to fix an even density $f$ on $[-1,1]$ and a function $\beta$ and evolve the martingale $G_t$ via the SDE
\begin{equation}
	\rmd G_t = \beta\big(G_t / \sqrt{t}\big)\, \rmd B_t\, \qquad t_0\le t\le 1 \,,
\end{equation}
for which $f$ is stationary, in the sense 
\begin{equation}\label{eq:stationary-f}
 \frac{G_t}{\sqrt t} \sim  f \,,\qquad ~\forall t_0 \le t \le 1\,. 
\end{equation}
This forces $|G_t|\le 1$ for all $t$ and $G_1  \stackrel{\rmd}{=} Z\sim f$. Moreover, since $b_t = \beta(G_t/\sqrt{t}) \stackrel{\rmd}{=} \beta(Z)$, the moment matching constraint~\eqref{eq:moment-matching-b} no longer depends on time.   
The stationarity constraint~\eqref{eq:stationary-f} fixes the choice of $\beta$:
\begin{equation}\label{eq:beta-choice}
\beta(z)^2 = \frac{\mathcal G_f(z)}{f(z)}\,,\qquad \mathcal G_f(z) := \int_z^1 yf(y)\rmd y\,,\quad -1<z<1\,. 
\end{equation} 
We remark that this choice satisfies the identity
\begin{equation}\label{eq:second-moment-beta}
	\E\beta(Z)^2 = \int_{-1}^1 \mathcal G_f(z) \rmd z = \E Z^2\,.
\end{equation}
Now for a fixed second moment parameter $v = \E G_1^2 = \E Z^2 =\E\beta(Z)^2$, we want to maximize
  \begin{equation}\label{eq:var-problem-f}
  	\E \beta(Z) = \int_{-1}^1 \sqrt{\mathcal G_f(z) f(z)}\,\rmd z
  \end{equation}
  over all even densities $f$ with second moment $v$. This maximizes the range of feasible $\alpha$ in~\eqref{eq:moment-matching-b}. 
One can obtain an implicit, ODE-based solution to this variational problem. For our purposes an equally good solution is obtained by maximizing the simpler proxy
\[\int_{-1}^1 \sqrt{f(z)}\, \rmd z\] 
obtained by applying the upper bound
\[\mathcal G_f(z) \le \mathcal G_f(0) = \frac{1}{2}\E |Z| \le \frac12 \sqrt{v}\,.\]
The solutions are explicit and given by the family $(f_\eta)_{\eta>0}$~\eqref{eq:profile-intro}. 
Using~\eqref{eq:beta-choice} we obtain a explicit expression for the control function: 
\begin{equation}\label{eq:beta-choice2}
\beta(z) = \beta_\eta(z) := \sqrt{\frac{1}{2\eta}(1-z^2)(\eta+1-z^2)}\,,\qquad b_t = \beta_\eta(G_t/\sqrt{t})\,.
\end{equation}
The parameter $\eta$ can then be chosen for a given $\alpha$ to satisfy the constraint~\eqref{eq:moment-matching-b}.   
Since $1-\E G_1^2 = 1-v_\eta$, $G_1$ is close to being binary for large $\alpha$, therefore the expected sign martingale run in the final stage 
\[G_{t} := \E\Big[\operatorname{sign}\Big(B_2-B_1+\Phi^{-1}\big((1+G_1)/2\big)\Big)\,\big|\,B_t\Big]\, \qquad 1 \le t\le 2\,,\]
with the Brownian path appropriately shifted, involves little uncertainty, and $\E[(G_2-G_1)^2] = 1-v_\eta$. This leaves very little room for the variance to accumulate.        

The next subsections describe the incremental AMP iteration and compare it with its continuum limit, the rounding step which produces a binary vector from the endpoint of a controlled martingale, and then proves Theorems~\ref{thm:random_matrix_discrepancy}-\ref{thm:large-alpha}.

\subsection{Incremental AMP}\label{subsec:incremental_amp}
We initialize the iteration with $u^0 = 0$, $g_0 \equiv \sqrt{\delta}$ for a fixed $\delta>0$. (In particular, $U^0 \equiv 0$.) We consider the (non-)linearities
\begin{align}\label{eq:choice_non_linearities}
m^{\ell} &:= g_{\ell}(u^0,\cdots,u^\ell) =  \sqrt{\delta} + \sum_{j=0}^{\ell-1} b_j(u^0,\cdots,u^j) (u^{j+1}-u^{j})\,, \qquad \ell\ge1, \\ 
M^0&:=a_{-1}V^0,\qquad
M^{\ell} :=a_{-1}V^0+\sum_{j=0}^{\ell-1} a_j (V^{j+1}-V^{j})\,,
\qquad \ell\ge1.
\end{align}
This choice is closely related to \cite{ams2020}.  
Here $a_{-1},a_0, a_1, \ldots$ are fixed nonzero scalar coefficients and $b_j:\R^{j+1}\to \R$, are strictly positive bounded differentiable functions, with $b_{-1}:=1$. 
In the notation of \eqref{eq:AMP0}, the corresponding triangular
coefficient array is
\begin{equation}\label{eq:incremental-coefficient-array}
 a_{0,0}=a_{-1},
 \qquad
 a_{t,s}=a_{s-1}-a_s\ (0\le s<t)\,, 
 \qquad
 a_{t,t}=a_{t-1}\quad(t\ge1)\,.
\end{equation}
This choice implies the following orthogonal increments property: 
\begin{proposition}\label{prop:orthogonal_increments}
	Let $(U^0,\ldots,U^{T+1},S^0,\ldots,S^T)$ be the limiting process with covariances specified by \eqref{eq:covU}--\eqref{eq:covS} with the above choices. Then for every $\ell\ge 0$, $U^{\ell+1} - U^{\ell}$ is independent of $U^0, \ldots, U^{\ell}$ and $S^{\ell+1} - S^{\ell}$ is free from $S^0, \ldots, S^{\ell}$. More specifically,
	\begin{align}\label{eq:covariance_independent_increments}
\E\big[ (U^{\ell+1} - U^{\ell}) U^j \big] = 0 \,,\qquad\qquad \taup\big( (S^{\ell+1} - S^{\ell}) S^j \big) = 0 \,~~~\forall\, j \le \ell\,.
\end{align}
Moreover,
 \begin{align}\label{eq:variance_increments}
\E\big[ (U^{1} - U^{0})^2 \big]
&= \alpha a_{-1}^2\taup((S^0)^2)=\alpha a_{-1}^2\delta,\notag\\
\E\big[ (U^{\ell+1} - U^{\ell})^2 \big]
&= \alpha a_{\ell-1}^2
\taup\big( (S^{\ell} - S^{\ell-1})^2 \big)\,,\\
\taup \big( (S^{\ell} - S^{\ell-1})^2 \big)
&= \E[b_{\ell-1}^2]\cdot\E\big[ (U^{\ell} - U^{\ell-1})^2 \big],
\qquad \ell\ge1\,.\notag
\end{align}
(In the above, $b_{\ell-1}$ is a shorthand for $b_{\ell-1}(U^0,\ldots,U^{\ell-1})$.)
\end{proposition}
\begin{proof}
We first prove that for every $\ell\ge 0$, $U^{\ell+1} - U^{\ell}$ is independent of $U^0, \ldots, U^{\ell}$ and $S^{\ell+1} - S^{\ell}$ is free from $S^0, \ldots, S^{\ell}$ and establish~\eqref{eq:covariance_independent_increments} by induction on $\ell$. The base case is $\ell=0$. Since $U^0 \equiv 0$ due to the choice of initialization, we obtain $\E\big[ (U^1- U^0) U^0 \big] = 0$ as $U^1$ is centered and independent from $U^0$. Furthermore as $U^0 = 0$, independence of $U^1 - U^0$ and $U^0$ follows. The base case for $S^1-S^0$ then follows from independence of $U^1-U^0 \equiv U^1$ and $U^0$, via an analogous argument as the inductive step for $S^{\ell+1}-S^\ell$ which is presented next.

For the inductive step, consider any $\ell \ge 1$ and any $j, 0 \le j \le \ell$. We suppose~\eqref{eq:covariance_independent_increments} holds for $\ell-1$ and all $j', 0 \le j' \le \ell-1$. The covariance definition~\eqref{eq:covU} gives 
\begin{align*}
\E\big[ (U^{\ell+1} - U^\ell) U^j \big] = \alpha \taup\big( (W^{\ell} - W^{\ell-1}) W^{j-1}\big) = \alpha \taup\big( a_{\ell-1} (S^\ell - S^{\ell-1}) W^{j-1} \big)\,.
\end{align*}
By the inductive hypothesis, $S^\ell - S^{\ell-1}$ is free from $S^0, \ldots, S^{\ell-1}$. Therefore as $S^\ell - S^{\ell-1}$ is centered it follows that 
\begin{align*}
\E\big[ (U^{\ell+1} - U^\ell) U^j \big] = 0\,.    
\end{align*}
As the $(U^1, U^2, \ldots )$ are jointly Gaussian, it follows that $U^{\ell+1} - U^\ell$ is independent from $U^0, \ldots, U^\ell$. This implies that 
\begin{align*}
\taup\big( (S^{\ell+1} - S^\ell) S^j \big) = \E\big[ (G^{\ell+1} - G^\ell) G^j \big] = \E\big[ b_\ell \cdot (U^{\ell+1} - U^\ell) \cdot g_\ell(U^0, \ldots, U^\ell) \big] = 0\,.
\end{align*}
As the $(S^0, S^1, \ldots)$ form a centered semicircular family, it follows that $S^{\ell+1}-S^\ell$ is free from $S^0, \ldots, S^\ell$. This completes the induction, proving~\eqref{eq:covariance_independent_increments}.

We now prove~\eqref{eq:variance_increments}. We first consider the first statement of~\eqref{eq:variance_increments}. By independence of $U^1-U^0$, $U^1$ from $U^0$ and the recursive definition~\eqref{eq:covU} of the covariances, we have 
\begin{align*}
\E\big[ (U^{1} - U^{0})^2 \big] = \E\big[ U^1(U^1 - U^0) \big] = \E\big[ (U^1)^2 \big] = \alpha a^2_{-1} \taup((S^0)^2)\,.
\end{align*}
For the second statement of~\eqref{eq:variance_increments}, considering $\ell \ge 1$ and as $U^{\ell+1}-U^\ell$ is independent from $U^\ell$, we have by~\eqref{eq:covU},
\begin{align*}
\E\big[ (U^{\ell+1} - U^{\ell})^2 \big] = \E\big[ (U^{\ell+1} - U^\ell) U^{\ell+1} \big] &= \alpha \taup\big( (W^\ell - W^{\ell-1}) W^\ell \big) \\
&= \alpha \taup\big( a_{\ell-1} (S^\ell - S^{\ell-1}) W^\ell \big) \\
&= \alpha a^2_{\ell-1} \taup\big( (S^\ell - S^{\ell-1})^2 \big)\,.
\end{align*}
Here we used that $S^\ell - S^{\ell-1}$ is centered and free from $S^0, \ldots, S^{\ell-1}$ and the definition of the $a_{t,s}$ from~\eqref{eq:incremental-coefficient-array}. Similarly for the last statement of~\eqref{eq:variance_increments}, as $S^{\ell} - S^{\ell-1}$ is free from $S^{\ell-1}$, we have by~\eqref{eq:covS},
\begin{align*}
\taup \big( (S^{\ell} - S^{\ell-1})^2 \big) = \taup\big( (S^{\ell} - S^{\ell-1}) S^\ell \big) &= \E\big[ (G^\ell - G^{\ell - 1}) G^\ell \big] \\
&= \E\big[  b_{\ell-1} \cdot (U^\ell - U^{\ell-1}) \cdot g_\ell(U^0, \ldots, U^\ell) \big] \\
&= \E\big[  b_{\ell-1}^2 \cdot (U^\ell - U^{\ell-1})^2 \big] = \E[b_{\ell-1}^2]\cdot\E\big[ (U^{\ell} - U^{\ell-1})^2 \big]\,.
\end{align*}
Here we used that $U^\ell - U^{\ell-1}$ is centered and independent from $U^0, \ldots, U^{\ell-1}$ and the definition of $g_\ell$ in~\eqref{eq:choice_non_linearities}. 
\end{proof}

Combining the above two variance relations, we impose the condition
\begin{align}\label{eq:variance_preserve_discretetime}
\alpha  a_{\ell-1}^2 \E[(b_{\ell-1})^2] = 1
\quad \forall \ell \ge 0\,.
\end{align}
It follows that $U^0,\ldots, U^\ell$ is a discretized Brownian motion at time intervals $\delta$, and $S^0,\ldots, S^\ell$ is a discretized free Brownian motion up to a time change:  
 \begin{equation}\label{eq:discretized_bm}
\E\big[ (U^{\ell} - U^{\ell-1})^2 \big] = \delta \,,\qquad\qquad 
\taup \big( (S^{\ell} - S^{\ell-1})^2 \big) = \E[b_{\ell-1}^2]\, \delta\,,~~~\forall\, \ell \ge 1\,.
\end{equation}

Next, we state the limits of $m^\ell$ and $\cA(m^\ell)$ implied by the above structure. 
\begin{proposition}\label{prop:m_A_limits}
	For every $\ell \ge 1$ we have 
	\begin{equation}\label{eq:m_A_limits}
	m^\ell \,\xrightarrow{\mathrm{emp}}\, \sqrt{\delta} + \sum_{j=0}^{\ell-1} b_{j} (U^{j+1} - U^j)\,,~~~
	\mbox{and}~~~\cA(m^\ell) \,\xrightarrow{\mathrm{str}}\, Y^{\ell} := S^{\ell} + \alpha \sum_{j=0}^{\ell-1} \E[b_{j}] a_{j-1} \big(S^j - S^{j-1}\big)\,.
	\end{equation} 
	In particular, under condition~\eqref{eq:variance_preserve_discretetime}, $Y^\ell$ is centered semicircular with  
	\begin{equation}\label{eq:semicircular_variance}
		\tau\big((Y^\ell)^2\big) = \sigma_\ell^2
		:=  \delta \E[b_{\ell-1}^2] + \delta \sum_{j=0}^{\ell-1}\big(1 + \alpha \E[b_j]a_{j-1}\big)^2 \,\E[b_{j-1}^2]\,,\qquad \|Y^{\ell}\| = 2\sigma_\ell\,.
	\end{equation}
\end{proposition}
\begin{proof}
The first statement is a direct consequence of state evolution, Theorem~\ref{thm:main-general}. The second statement can be extracted from the AMP iteration:
\begin{equation}\label{eq:A0}
	\cA(m^\ell) = V^{\ell} + \alpha_d \sum_{j=1}^\ell d_{\ell,j} M^{j-1}\,,
\end{equation}
  with 
\begin{align}
d_{\ell,j} &= \frac{1}{n} \sum_{i=1}^n \partial_{j}g_{\ell}(u_i^0,\ldots,u_i^\ell)\\
&=
\begin{cases}
\Av_n(b_{\ell-1}) & \mbox{if }~~ j=\ell\,,\\
\Av_n(b_{j-1}) - \Av_n(b_j) + \Delta_j & \mbox{if }~~ j\le\ell-1\,,
\end{cases}
\end{align}
where, for any function $b : \R^{q+1}\to \R$, we used the notation 
\[\Av_n(b) := \frac{1}{n} \sum_{i=1}^n b(u_i^0,\ldots,u_i^q)\,,~~~\mbox{and}~~~ \Delta_j := \sum_{s = j}^{\ell-1} \Av_n\Big(\partial_j b_s \cdot (u^{s+1} - u^s)\Big)\,.\]

From~\eqref{eq:A0} we obtain
\begin{align}\label{eq:A1}
\cA(m^\ell) &= V^{\ell} + \alpha_d  \Av(b_{\ell-1}) M^{\ell-1} + \alpha_d \sum_{j=1}^{\ell-1} \big(\Av_n(b_{j-1}) - \Av_n(b_j) + \Delta_j\big) M^{j-1}\,,\notag\\
&= V^{\ell} + \alpha_d \sum_{j=0}^{\ell-1} \Av_n(b_{j}) \big(M^j - M^{j-1}\big) + \alpha_d \sum_{j=1}^{\ell-1}  \Delta_j M^{j-1}\,,\notag\\
&= V^\ell + \alpha_d \sum_{j=0}^{\ell-1} \Av_n(b_{j}) a_{j-1} \big(V^j - V^{j-1}\big) + \alpha_d \sum_{j=1}^{\ell-1}  \Delta_j M^{j-1}\,,
\end{align}
where we used $V^{-1}=M^{-1}=0$ and the relation $M^{j}-M^{j-1} = a_{j-1}(V^{j}-V^{j-1})$. 
State evolution applied to the pseudo-Lipschitz functions $b_j, \partial_s b_j$ gives 
\begin{align*}
\Av_n(b_j) \xrightarrow{p} \E[b_j] \,,~~~\mbox{and}~~~ \Delta_j\xrightarrow{p} 0 \,,
\end{align*}
since by Proposition~\ref{prop:orthogonal_increments}, for every $s\ge j$, $U^{s+1}-U^s$ is centered and independent of $\partial_jb_s(U^0,\ldots,U^s)$. Moreover, since $\alpha_d \to \alpha$ and $M^{j-1}$ converges strongly, the last term in~\eqref{eq:A1} converges strongly to $0$. Therefore, by the strong convergence statement in Theorem~\ref{thm:main-general} we have         
\begin{align}\label{eq:A2}
\cA(m^\ell) \,&\xrightarrow{\mathrm{str}}\, S^{\ell} + \alpha \sum_{j=0}^{\ell-1} \E[b_{j}] a_{j-1} \big(S^j - S^{j-1}\big)\,,\notag\\
&= S^{\ell} - S^{\ell-1} +  \sum_{j=0}^{\ell-1} \big(1+\alpha\E[b_{j}] a_{j-1}\big) \big(S^j - S^{j-1}\big)\,,
\end{align}
where we used $S^{-1}=0$.
The convergence of the variance follows by orthogonality and~\eqref{eq:discretized_bm}. 
\end{proof}

\subsection{Proof of Theorem~\ref{thm:random_matrix_discrepancy}}
\label{sec:diffusion}

We need to choose the free parameters $a_j,b_j$ to design the algorithm. Mirroring condition~\eqref{eq:variance_preserve_discretetime} we set
\begin{equation}\label{eq:choice_a}
a_j = - \frac{1}{\sqrt{\alpha \E[b_j^2]}}
\qquad \forall j \ge -1\,.	
\end{equation}
The minus sign is chosen in view of minimizing the variance achieved in~\eqref{eq:semicircular_variance} which now becomes
\begin{equation}\label{eq:variance_1}
\sigma_\ell^2  = \delta \E[b_{\ell-1}^2] + \delta \sum_{j=0}^{\ell-1}\Big(\sqrt{\E[b_{j-1}^2]} - \sqrt{\alpha} \E[b_j]\Big)^2 \,.
\end{equation}
On the other hand $m^\ell$ converges empirically (Proposition~\ref{prop:m_A_limits}) to the martingale  
\begin{equation}\label{eq:Muhat}
	G^\ell = \sqrt{\delta} + \sum_{j=0}^{\ell-1} b_j (U^{j+1}-U^j)\,,
\end{equation} 
which we want to steer towards $\{-1,+1\}$ at some terminal time. 
 Let $N=2^k$, $\delta=T/N$, $t_j=j\delta$, and let $(B_t)_{t \ge 0}$ be a Brownian motion on $(\Omega,\mathcal F, \Prob)$, and consider the coupling $U^j = B_{t_j}$, $0 \le j \le N$. Let $\mathbb F = (\mathcal F_t)_{t\ge0}$ be the filtration generated by $B$. For a ($\mathbb F$-predictable) control $b$, recall the notation
\[ J_\alpha(b)=
\int_0^T \Big(\sqrt{\E b_s^2} - \sqrt{\alpha}\E b_s\Big)^2 \rmd s\,. \]

We first need the following approximation results:
\begin{lemma}\label{lem:cylinders}
If $b\ge0$ is predictable and square-integrable on $[0,T]\times\Omega$,
there are functions
\[
 b_{k,j}\in C_b^\infty(\R^{j+1})\,,\qquad
 b_{k,j}\ge\nu_k>0\,,\qquad 0\le j<N\,,
\]
such that, with $b^{(k)}_t=b_{k,j}(B_{t_0},\ldots,B_{t_j})$ for $t_j\le t<t_{j+1}$, we have
\begin{equation}\label{eq:approx-controls}
 \epsilon_k:=\int_0^T \E\big[(b^{(k)}_t-b_t)^2\big]\rmd t \xrightarrow[k \to \infty]{} 0\,.
\end{equation}
\end{lemma}
\begin{proof}
Fix $k\ge 0, N=2^k,\delta = T/N$. 
We approximate $b$ by the following projections 
\begin{equation}\label{eq:projection}
 F_{k,j}
 :=\frac{1}{\delta}\E\left[\int_{t_j}^{t_{j+1}}b_s\rmd s
 \ \Big|\ B_{t_0},\ldots,B_{t_j}\right]\,,
 \qquad
 P_kb(t)=F_{k,j}\,,\quad t_j\le t<t_{j+1}\,.
\end{equation}
By construction, the process $P_kb$ is the orthogonal projection of $b$ on the subspace of elementary adapted processes $\sum_{j} \xi_j \mathbf{1}_{[t_{j},t_{j+1})}$, where $\xi_j \in L^2(\mathcal F_{t_j})$. We briefly argue that the union over $k$ of such subspaces is dense in $L^2([0,T]\times\Omega)$. \cite[Definition~5.2 and Proposition~5.3]{le2016brownian} guarantees that the processes $\{\sum_{j=0}^{n-1} \xi_j \mathbf{1}_{(s_j,s_{j+1}]}:  n\ge 1, \xi_j \in L^2(\mathcal F_{s_j}) , 0=s_0<s_1<\ldots<s_n=T\}$ are dense in $L^2([0,T]\times\Omega)$. It remains to restrict to our fixed dyadic partition of $[0,T]$. Fix $\xi\in L^2(\mathcal F_s)$ and let
\[
 \mathcal H_k(s)
 =\sigma(B_{jT2^{-k}}:0\le jT2^{-k}\le s)\,.
\]
By almost-sure continuity of Brownian motion paths we have $\sigma(\bigcup_k\mathcal H_k(s))=\mathcal F_s$ up to null sets.
Consequently, classical martingale convergence results \cite[Theorems~4.6.8 and~4.4.6]{durrett2019probability} imply
\[
 \xi_k:=\E[\xi\mid\mathcal H_k(s)]\longrightarrow\xi
 \quad\text{in }L^2(\Omega)\,.
\]
Now for an elementary term $\xi\ind_{(s,t]}$, we round $s$ upward and $t$
downward to the dyadic grid $\{jT2^{-k}\}$, obtaining $s_k,t_k$. The process $\xi_k\ind_{(s_k,t_k]}$ is grid-predictable
and converges to $\xi\ind_{(s,t]}$ in $L^2([0,T]\times \Omega)$, thereby establishing the desired density. 
We deduce that $P_kb\to b$ in $L^2([0,T]\times \Omega)$ as $k \to \infty$. 

Next, each $F_{k,j}$ is nonnegative and is a Borel function $f_{j,k}\ge 0$ of $B_{t_0},\ldots,B_{t_j}$. To obtain the functions $b_{j,k}$, we truncate $f_{j,k}$ to not exceed a cutoff $R_k \to \infty$, then convolve the result by a positive compactly supported $C^\infty$ kernel (e.g. the Gaussian kernel with variance $2^{-k}$), then add a positive constant $\nu_k\to 0$. Such an approximation still converges in $L^2$ to $b$. 
\end{proof}

\begin{lemma}\label{lem:approximations}
Let $b, b^{(k)}$ as in the previous lemma.  Write 
 \begin{align}
 G^N &:= \sqrt{\delta} + \sum_{j=0}^{N-1} b_{k,j}(B_0,\ldots,B_{t_j})(B_{t_{j+1}}-B_{t_j}) = \sqrt{\delta} + \int_0^T b_t^{(k)}\rmd B_t\,,\label{eq:martingale_3.7}\\
 	\sigma_{N}^2 &:= \delta \E[b_{k,N-1}^2] + \delta \sum_{j=0}^{N-1}\Big(\sqrt{\E[b_{k,j-1}^2]} - \sqrt{\alpha} \E[b_{k,j}]\Big)^2\,.\label{eq:variance_3.7}
 \end{align}
Then we have
\begin{equation}\label{eq:message-varinace-error}
 \E\big[(G^N-G_T)^2\big] ~\vee~
 \left|\sigma_{N}-\sqrt{J_\alpha(b)}\right|
 \xrightarrow[k \to \infty]{} 0\,.
\end{equation}
\end{lemma}

\begin{proof}
By~\eqref{eq:martingale_3.7} and It\^o's isometry and Lemma~\ref{lem:cylinders},
\[\E\big[(G^N-G_T)^2\big] = \delta + \epsilon_k \xrightarrow[k \to \infty]{} 0\,.\]
Next, write 
\[\mu(t) = \E b_t\,,\quad \mu_k(t) = \E b^{(k)}_t\,,\quad \rho(t) = \sqrt{\E b_t^2}\,,\quad \rho_k(t) = \sqrt{\E (b^{(k)}_t)^2}\,.\]
By Jensen's inequality and the reverse triangle inequality we have
\[\|\mu - \mu_k\|_{L^2([0,T])} \le \sqrt{\epsilon_k}\,,\qquad \|\rho - \rho_k\|_{L^2([0,T])} \le \sqrt{\epsilon_k}\,,\]  
where $\epsilon_k  = \|b^{(k)}-b\|_{L^2([0,T] \times\Omega)}^2 \to 0$ by Lemma~\ref{lem:cylinders}. Now let  
\begin{equation*}
 \rho_k^-(t)=
 \begin{cases}
 1,&0\le t<\delta,\\
 \rho_k(t-\delta),&\delta\le t\le H.
 \end{cases}
\end{equation*}
and define $\rho^-$ similarly. Then
\[\|\rho_k^- - \rho\|_{L^2([0,T])} \le \|\rho_k^- - \rho^-\|_{L^2([0,T])}+\|\rho^- - \rho\|_{L^2([0,T])}\,.\]
The first term on the right-hand side is bounded by $\sqrt{\epsilon_k}$. As for the second term, we extend $\rho$ by $\rho(t)=1$ for $-1<t<0$ and $\rho(t)=0$ on $(-\infty,-1]\cup [1,+\infty)$. Then, by approximating $\rho$ in $L^2$ with $C^1$ functions with compact support we can argue that 
\[\|\rho^- - \rho\|_{L^2([0,T])}^2 = \int_{\R} (\rho(t)-\rho(t-\delta))^2 \rmd t \xrightarrow[\delta \to 0]{} 0\,.\]
We conclude that 
\[\|\rho_k^- - \rho\|_{L^2([0,T])}\xrightarrow[k \to \infty]{} 0\,.\]   
Next, let $\mu_{k,j}, r_{k,j}$ be the first and second moments of $b_{k,j}$. Then~\eqref{eq:variance_3.7} becomes
\[
 \sigma_N^2 =\delta r_{k,N-1}
 +\|\rho_k^--\sqrt\alpha\,\mu_k\|_{L^2([0,T])}^2\,.
\]
The first term tends to zero:
\[\sqrt{\delta r_{k,N-1}} =  \|b^{(N-1)}\|_{L^2([T-\delta,T] \times\Omega)} \le \|b\|_{L^2([T-\delta,T] \times\Omega)}+\sqrt{\epsilon_k} \xrightarrow[k \to \infty]{} 0\,.\]
For the second term,
\[\Big|\|\rho_k^--\sqrt\alpha\,\mu_k\|_{L^2([0,T])} - \|\rho-\sqrt\alpha\,\mu\|_{L^2([0,T])}\Big| \le \|\rho_k^--\rho\|_{L^2([0,T])} + \sqrt{\alpha}\|\mu_k -\mu\|_{L^2([0,T])}\,.\]
Both terms tend to zero by the previous arguments. Therefore, $|\sigma_N - \sqrt{J_\alpha(b)}|$ tends to zero.  
\end{proof}

\begin{proof}[Proof of Theorem~\ref{thm:random_matrix_discrepancy}]
Fix $k$, $T$ and $N = 2^k, \delta = T/N$. Fix a regular approximation $b_{k,j}$ guaranteed by Lemma~\ref{lem:cylinders}. Let the incremental map 
\begin{equation}\label{eq:terminal-incremental-map}
g_{\ell}(u^0,\ldots,u^\ell) := \sqrt{\delta} + \sum_{j=0}^{\ell-1} b_{k,j}(u^0,\ldots,u^j)(u^{j+1}-u^j)\,,\quad \ell \le N\,,
\end{equation}
with $g_{0}\equiv\sqrt{\delta}$.

By Lemma~\ref{lem:cylinders}, these maps are admissible:
for every $\ell\le N$, $g_{\ell} \in C^1(\R^{\ell+1})$ and
\[
 \partial_0 g_{\ell},\ldots,
 \partial_j g_{\ell}\in\PL_2\,,\qquad
 \partial_\ell g_{\ell} =b_{k,\ell}>\nu_k>0\,.
\]
Thus Assumptions~\ref{ass:g} and \ref{ass:transversality} hold. 
Next, $|a_{j}| = 1/\sqrt{\alpha r_j}$, $r_j =\E b_{k,j}(B_{\le t_j})^2$ is bounded from above and below. Therefore, Assumption~\ref{ass:a1} holds.
Let $h = \operatorname{sign} \circ g_N$. Conditionally on $B_0,\ldots,B_{(N-1)\delta}$,
\[
 G^{N}=G^{N-1}+b_{N-1}(B_{\le(N-1)\delta})
 (B_{N\delta}-B_{(N-1)\delta})
\]
has a Gaussian distribution with positive variance. Thus $\Prob(G^N=0)=0$ and our choice of $h$ satisfies the null-discontinuity Assumption~\ref{ass:zero-discontinuity}.

Now by the state evolution Theorem~\ref{thm:main-general},  we have for every $\eta>0$,
\begin{equation}
 \Prob\left(
 \bigl\|\cA(\widehat x - m^N)\bigr\|_{\op}
 >
 2(1+\sqrt\alpha)\,
 \sqrt{e_{N}}
 +\eta
 \right)
 \xrightarrow[d \to \infty]{}  0\,,
\end{equation}
where
\begin{equation}\label{eq:terminal-sign-vector}
 m_i^N = g_N(u^0,\ldots,u^N)\,,\qquad \widehat x_i:=\operatorname{sign}(m_i^N)\,,\qquad 1 \le i \le n\,,
\end{equation}
\begin{equation}\label{eq:ideal-rounding-error}
 e_{N}
 :=\E[(H - G^N)^2]
 =\E[(1-|G^N|)^2]\,,\qquad 
 H:=\operatorname{sign}(G^N)\,.
\end{equation}
By Lemma~\ref{lem:approximations} and the assumption $G_T \in \{-1,+1\}$ a.s., we have by the triangle inequality
\[e_{N} \le  \E\big[(G^{N}-G_T)^2\big] \xrightarrow[k \to \infty]{} 0\,.\]
Next, by Proposition~\ref{prop:m_A_limits}, $\Prob(\cA(m^N)> 2\sigma_N) \to 0$, and again by Lemma~\ref{lem:approximations}, 
\[\left|\sigma_{N}-\sqrt{J_\alpha(b)}\right| \xrightarrow[k \to \infty]{} 0 \,.\]
This allows us to conclude that for every $\eta>0$, letting $k$ be large enough,
 \begin{equation}\label{eq:finite-grid-round}
 \Prob\left(\norm{\cA(\widehat x)}_{\op}
 >2\sqrt{J_\alpha(b)} +\eta\right)
 \xrightarrow[d \to \infty]{} 0\,.
\end{equation}
\end{proof}

\subsection{Two implementations}\label{sec:implementations}
We first implement a simple control based on the heat equation.  
Let $g: [0,1]\times \R \to \R$ given by  
\begin{equation}\label{eq:heat_equation_solution}
	g(t,x) = \E\big[\mathrm{sign}(B_1) \mid  B_t=x\big] = 2 \Phi\Big(\frac{x}{\sqrt{1-t}}\Big)-1\,.
\end{equation}
Recall that $\Phi$ is the cumulative distribution function of the standard normal distribution.  
$g$ has terminal value $g(1,x) = \mathrm{sign}(x)$ and satisfies the heat equation 
\begin{equation}\label{eq:heat_equation}
	\partial_t g + \frac{1}{2}\partial_{x}^2 g = 0\,,\qquad (t, x) \in [0,1) \times \R \,.
\end{equation}
 Now let  
 \begin{equation}\label{eq:G_t}
 	G_t := g(t,B_t)\,.
 \end{equation}
It can be verified by It\^{o}'s formula that 
\begin{align}\label{eq:heat-control-correct}
	\rmd G_t &= b_t \,\rmd B_t\,, \qquad b_t = b(t,B_t) \,,\\
	b(t,x) &:= \partial_x g(t,x) = \sqrt{\frac{2}{\pi(1-t)}} \,\exp\Big\{-\frac{x^2}{2(1-t)}\Big\}\,. \notag
\end{align}
The control $(b_t)_{t\in[0,1]}$
is predictable and square-integrable on $\Omega \times [0,1]$. From Theorem~\ref{thm:random_matrix_discrepancy}, the asymptotic operator norm achieved by the algorithm is $J_\alpha(b)$ which we now compute.
 
\begin{lemma}
With the above $b$,  
\[J_\alpha(b) = 1-\frac{2\sqrt\alpha}{\pi}
 B\left(\frac12,\frac34\right)
 +\frac{2\alpha}{\pi}\,,
 \]
 where 	
 \[
 B\left(\frac12,\frac34\right) = 2\int_0^1(1-t^2)^{-1/4}\,\rmd t
 \,.
\]
In particular $J_\alpha(b)<1$ if $\alpha < \alpha_* = B(1/2,3/4)^2$.
\end{lemma}

\begin{proof} 
By Gaussian  integration we have
\begin{equation}\label{eq:heat-control-moments}
 \E[b(t,B_t)]=\sqrt{\frac{2}{\pi}}:=\mu\,,
 \qquad
 \E[b(t,B_t)^2]=\frac{2}{\pi\sqrt{1-t^2}}\,.
\end{equation}
Then
\begin{align*}
 J_\alpha(b) &= \int_0^1 \Big(\sqrt{\frac{2}{\pi}}(1-t^2)^{-1/4}-\sqrt{\alpha}\sqrt{\frac{2}{\pi}}\Big)^2\,\rmd t\\
 &= \frac{2}{\pi}\int_0^1 (1-t^2)^{-1/2}\rmd t - \frac{4\sqrt\alpha}{\pi}\int_0^1 (1-t^2)^{-1/4}\rmd t +\frac{2\alpha}{\pi}\,. 
\end{align*}
The first term in the above expression is equal to 1. 
\end{proof}

Next, we design a three-stage control achieving $J_\alpha(b)= O(1/(\alpha \log \alpha))$ for large $\alpha$. 

Recall the density  
\[ f_\eta(z)=\frac{c_\eta}{(\eta+1-z^2)^2}\,,\quad -1\le z\le1\,,\quad \eta>0\,,\]
and identically zero otherwise.

\paragraph{Phase I:} Let $0<t_0<1$ and $\eta>0$. We want a control such that $G_{t_0}/\sqrt{t_0} \sim f_\eta$:
\[G_t = \sqrt{t_0} \, \E\Big[ F_\eta^{-1}\big(\Phi(B_{t_0}/\sqrt{t_0})\big)\,\big|\, B_t\Big]\,,\qquad 0\le t\le t_0\,,\]
where $F_\eta$ is the cumulative distribution function of $f_\eta$: $F_\eta(z) = \int_{-1}^z f_\eta(y)\rmd y$. Then we have
\[\rmd G_t =\,\partial_x p(t,B_t)\, \rmd B_t\,,\qquad p(t,x) = \E F_\eta^{-1}\big(\Phi(\sqrt{1-t/t_0}Z+ x/\sqrt{t_0})\big)\,,~~~Z \sim N(0,1)\,.\]
The control $b_t = \sqrt{t_0}\partial_x p(t,B_t)$ is admissible as per our assumptions, and $G_0=0$, $|G_{t_0}|\le \sqrt{t_0}\le 1$. We remark that this phase can be implemented in one discrete AMP step, by choosing $t_0=\delta$, $g_0 = \sqrt{\delta}$, $g_1(u^0,u^1) = \sqrt{\delta} F_\eta^{-1}\circ \Phi(u^1/\sqrt{\delta})$. We are writing this phase in continuous time so it fits our general control framework.      

The variance contribution of this phase can be bounded by It\^{o}'s isometry as 
\begin{equation}\label{eq:variance-phase-I}
J_{\alpha}(b_{|[0,t_0]}) = \int_0^{t_0} \left(\sqrt{\E b_t^2} -\sqrt{\alpha}\E b_t \right)^2 \rmd t \le (1+\alpha) \int_0^{t_0}\E b_t^2 \rmd t \le (1+\alpha) t_0 v_\eta\,,  
\end{equation}  
where $v_\eta = \E Z_\eta^2, Z_\eta \sim f_\eta$.

\paragraph{Phase II:} This is the main phase where $G_t/\sqrt{t}$ is forced to have a stationary law $f_\eta$. On $[t_0,1]$, consider the evolution  
\begin{equation}\label{eq:main-SDE}
 \rmd G_t=\beta_\eta(G_t/\sqrt t)\rmd B_t\,,\qquad \beta_\eta^2(z) = \frac{\int_z^1 yf_\eta(y)\rmd y}{f_\eta(z)}=\frac{1}{2\eta}(1-z^2)(\eta+1-z^2)\,,
\end{equation}
while imposing continuity with the previous phase at $t_0$ .  
\begin{lemma}
The above SDE has a unique strong solution for every initial condition $G_{t_0}$. Moreover, $f_\eta$ is stationary for this evolution: if $G_{t_0}/\sqrt{t_0} \sim f_\eta$ then $G_t/\sqrt{t} \sim f_\eta$ for all $t\ge t_0$.
\end{lemma}
\begin{proof}
Consider the time change
\[
 Y_s=\frac{G_{t_0e^s}}{\sqrt{t_0e^s}}\,,\qquad
 W_s=\int_{t_0}^{t_0e^s}u^{-1/2}\rmd B_u\,.
\]
Then $(W_s)_{s\ge 0}$ is another standard Brownian motion, and by It\^o's formula, $Y_s$ satisfies the   SDE
\[\rmd Y_s = -\frac{1}{2} Y_s \rmd s + \beta_\eta(Y_s)\rmd W_s\,.\]
Since the coefficients are bounded and locally Lipschitz on $(-1,1)$, the above SDE has a unique path-continuous maximal solution of every initial condition $|Y_0|=|y|<1$, up to the exit time 
\[\tau = \lim_{r\to 1}\tau_r \,,\qquad \tau_r = \inf\{s \ge 0 : |Y_s|\ge r\}\,,\] 
see \cite[Theorem 6.9]{miller2016stochastic}. We need to show that $\tau = \infty$ a.s. Let $V(z) = -\log(1-z^2)$, and let $S>0$ be fixed. By It\^o's formula, 
\[V(Y_{\tau_r \wedge S}) = V(y)+ \int_0^{\tau_r \wedge S}\frac{1}{2}\Big(-Y_sV'(Y_s)+ \beta_\eta(Y_s)^2V''(Y_s)\Big) \rmd s + \int_0^{\tau_r \wedge S} V'(Y_s)\beta_\eta(Y_s)\rmd W_s\,.\]
A simple calculation shows that the integrand in the first integral above is bounded by $C_\eta = \frac{\eta+2}{2\eta}$ uniformly in $Y_s$. Therefore,
\[\E V(Y_{\tau_r \wedge S}) \le V(y)+ C_\eta S\,.\]
On the event $\{\tau_r <S\}$, $Y_{\tau_s}= r$ by continuity of $Y$. Since $V(z) \ge 0$,
\[\Prob(\tau_r<S) \le \frac{V(y)+C_\eta S}{V(r)}\,.\]
Letting $r\to 1$ implies $\Prob(\tau < S)=0$ for every $S$. Therefore $\tau=\infty$ a.s. We conclude uniqueness of a global path-continuous strong solution.  
Next, assuming $Y_t$ has a density, the latter satisfies the PDE
\[\partial_t p_t = \frac{1}{2} \big(\partial_y (yp_t) + \partial_{y}^2(\beta_\eta^2 p_t)\big)\,.\]     
Since by definition,
\[\beta_\eta^2(z) = \frac{\int_z^1 yf_\eta(y)\rmd y}{f_\eta(z)}\,, \quad -1<z<1\,,\]
the density $p_t = f$ is stationary for the above: $z f_\eta(z) + (\beta_\eta^2 f_\eta)'(z)=0$. So if $Y_0 \sim f_\eta$ then $Y_s$ has density $f_\eta$ for all subsequent times. 
\end{proof}

In particular $G_1 \sim f_\eta$, $\E G_1^2 = v_\eta$, and with $b_t = \beta_\eta(G_t/\sqrt{t})$, $t\in (t_0,1]$,
\begin{equation}\label{eq:variance-phase-II}
	J_\alpha(b_{|[t_0,1]}) = \int_{t_0}^1 \big(\sqrt{\E b_t^2} -\sqrt{\alpha}\E b_t \big)^2 \rmd t = (1-t_0) \Big(\sqrt{\E \beta_\eta(Z_\eta)^2} -\sqrt{\alpha} \E\beta_\eta(Z_\eta)\Big)^2 \,.
\end{equation}
Now we choose $\eta$ so as to make this variance vanish, justified by the following lemma. First let
\[\mu_\eta=\E\beta_\eta(Z_\eta)\,,\]
and notice that the expression of $\beta_\eta$ in~\eqref{eq:main-SDE} implies 
\[\E \beta_\eta(Z_\eta)^2 = \E Z_\eta^2 = v_\eta\,.\]
\begin{lemma}\label{lem:moment-matching} 
For all $\alpha > \alpha_0 = 32/(3\pi^2)$, there exists $\eta = \eta_\alpha>0$ such that 
\begin{equation}\label{eq:mu-squared-v} 
\alpha \mu_\eta^2=v_\eta \,.	
\end{equation}
Consequently, this phase has no variance accumulation: \[J_\alpha(b_{|[t_0,1]})=0\,.\]
Additionally, $\eta$ can be selected so that $\eta_\alpha \to 0$ as $\alpha \to \infty$. 
\end{lemma}
The proof of this lemma is deferred to after we present Phase III.

\paragraph{Phase III:} In this last stage we run the expected sign martingale. Let
\[\chi(z) := \Phi^{-1} \Big(\frac{1+z}{2}\Big)\,, \]
\begin{equation}\label{eq:terminal-target}
 F:=\operatorname{sign}(\chi(G_1)+B_2-B_1)\,,\qquad
 G_t=\E[F\mid \mathcal F_t]\,,\quad 1\le t\le 2\,.
\end{equation}
This construction makes $G_t$ continuous at $t=1$, and $G_2 \in \{-1,+1\}$ a.s. Orthogonality of increments and It\^o's formula yield
\begin{align} 
G_t &= 2\Phi\Big(\frac{X_t}{\sqrt{2-t}}\Big)-1\,,\qquad X_t = B_t - B_1 + \chi(G_1)\,,\\
\rmd G_t &= 2\varphi_{2-t}(X_t)\rmd B_t\,,\qquad 1 \le t \le 2\,,
\end{align}  
 where we recall that $\varphi_{\sigma^2}$ is the p.d.f.\ of $N(0,\sigma^2)$. Using $b_t = 2\varphi_{2-t}(X_t), t\in (1,2]$,
 let 
 \[h_\eta = 2\E\varphi\circ \chi(Z_\eta)\,,\qquad I_\eta = \int_0^1 \sqrt{\E b_{1+s}^2}\rmd s\,.\]
 \begin{lemma}\label{lem:variance-phase-III}
 The variance accumulated in $[1,2]$ takes the expression
  \begin{equation}\label{eq:variance-phase-III}
	J_\alpha(b_{|[1,2]}) = 1-v_\eta -2\sqrt{\alpha}h_\eta I_\eta + \alpha h_\eta^2  \,.
\end{equation}
\end{lemma}
\begin{proof}
We have
	\[J_\alpha(b_{|[1,2]}) = \int_{1}^2 \E b_t^2 \rmd t -2\sqrt{\alpha}\int_1^2\sqrt{\E b_t^2}\E b_t  \rmd t + \alpha \int_{1}^2 (\E b_t)^2\rmd t\,. \]
Since $|G_2| =1,  \E G_1^2 = v_\eta$, we obtain by It\^o's isometry
\[1-v_\eta = \E (G_2-G_1)^2 = \int_1^2 \E b_t^2 \rmd t\,.\]
Next, Gaussian integration gives the conditional expectation
\begin{align*}
 \E[b_{t}\mid\mathcal F_1]= 2\E[\varphi_{2-t}(X_{t})\mid\mathcal F_1]= 2\varphi \circ \chi(G_1)\,.
\end{align*}
 	Since $G_1 \stackrel{\rmd}{=}Z_\eta \sim f_\eta$, we obtain $\E b_{t} = h_\eta$, and this yields the claimed expression. 
\end{proof}

\begin{lemma}\label{lem:profile-asymptotics}
Let $L=\log(1/\eta)$. As $\eta\to 0$,
\begin{equation}\label{eq:profile-asymptotics}
 1-v_\eta\sim\eta L\,,\qquad
 \mu_\eta\sim\sqrt{\eta/2}\,L\,, \qquad h_\eta\sim\frac{\sqrt2}{3}\eta L^{3/2}\,.
\end{equation}
The ratio $\mu_\eta^2/v_\eta$ is continuous, tends to zero as $\eta \to 0$,
and tends to $3\pi^2/32$ as $\eta\to\infty$.
\end{lemma}
\begin{proof}
Elementary computations imply 
\begin{equation}
 c_\eta^{-1}
 =\frac1{\eta A}
 +\frac{1}{2A^{3/2}}\log\Big(\frac{\sqrt{A}+1}{\sqrt{A}-1}\Big)\,,\quad 
 v_\eta
 =A-\frac{c_\eta}{\sqrt A}
 \log\Big(\frac{\sqrt{A}+1}{\sqrt{A}-1}\Big)\,,\quad A=1+\eta\,.\label{eq:exact-v}
\end{equation}
This yields	$c_\eta\sim\eta$. By the change of variables $r=1-z^2$ in the integrals defining $v_\eta$ and $\mu_\eta$,           
\begin{align}
 1-v_\eta=c_\eta\int_0^1
 \frac{r}{(\eta+r)^2\sqrt{1-r}}\rmd r\,,\qquad 
 \mu_\eta=\frac{c_\eta}{\sqrt{2\eta}}\int_0^1
 \frac{\sqrt r}{(\eta+r)^{3/2}\sqrt{1-r}}\rmd r\,.
 \label{eq:profile-integrals}
\end{align}
Each integral is $L+O(1)$: split the integrals at $\eta$ and at
$1/2$; the first and last pieces are bounded, and on the middle
piece replacing the integrand by $1/r$ produces a bounded
integral error. 
Next, with $\psi(x) = 2\varphi \circ \Phi^{-1}((1+x)/2)$ and the change of variables $1-|z|=x$, one can check that
\[h_\eta = \int_0^1 \psi(1-x)p(x)\rmd x\,,\qquad p(x)=\frac{2c_\eta}{(\eta+2x-x^2)^2}\,. \]%
 We observe that 
 \[p(\eta^y) \sim c_\eta/(2\eta^{2y})\,,\quad 0<y<1\,,\qquad  p(\eta^y) \sim 2c_\eta/\eta^2\,, \quad y>1\,,\] 
 and $\psi(1-x) \sim x\sqrt{2\log(1/x)}$ for $x\to 0$. After the change of variables $x=\eta^y=e^{-Ly}$, and dividing by $c_\eta L^{3/2}$, the integrand converges to $\sqrt{y/2}\mathbf{1}\{y\in(0,1)\}$. Dominated convergence allows us to obtain
 \[\frac{h_\eta}{c_\eta L^{3/2}} \sim \int_0^1 \sqrt{y/2}\rmd y = \frac{\sqrt{2}}{3}\,.\]
Next, continuity with respect to $\eta>0$ follows from the integral formulas.
As $\eta\to\infty$, $f_\eta(z)\to1/2$ and
$\beta_\eta(z)^2\to(1-z^2)/2$ uniformly. Hence
\[
 v_\eta\to\frac{1}{3}\,,\qquad
 \mu_\eta\to\frac{1}{2\sqrt{2}}\int_{-1}^1\sqrt{1-z^2}\rmd z
 =\frac{\pi}{4\sqrt2}\,,
\]
which yields the stated limiting ratio.
\end{proof}

\begin{proof}[Proof of Lemma \ref{lem:moment-matching}]
	Since $\mu_\eta^2/v_\eta$ is continuous and has range $(0,3\pi^2/32)$, there exists a solution $\eta>0$ to~\eqref{eq:mu-squared-v} for any $\alpha>32/(3\pi^2)$ by the intermediate value theorem. We select the smallest such solution $\eta = \eta_\alpha>0$. Moreover for large $\alpha$, such a solution would have to satisfy $\mu_\eta^2/v_\eta \to 0$, therefore by continuity, $\eta_\alpha \to 0$.   
\end{proof}

\begin{proof}[Proof of Theorem \ref{thm:large-alpha}]
Let's take $\eta = \eta_\alpha$ as in Lemma~\ref{lem:moment-matching}. Using the asymptotic equivalents in~\eqref{eq:profile-asymptotics} we obtain  
\[\eta_\alpha \sim\frac{2}{\alpha (\log \alpha)^2} \,,\qquad 1-v_\eta \sim \frac{2}{\alpha\log \alpha}\,,\qquad h_\eta \sim \frac{2\sqrt{2}}{3}\frac{1}{\alpha\sqrt{\log \alpha}}\,.\] 
By Lemma~\ref{lem:variance-phase-III},
\[
\begin{aligned}
J_\alpha(b_{|[1,2]}) &=  1-v_\eta -2\sqrt{\alpha}h_\eta I_\eta + \alpha h_\eta^2\,,\\	
\mbox{and}&\quad 1-v_\eta + \alpha h_\eta^2 \sim \frac{2+8/9}{\alpha \log \alpha}\,,\qquad \alpha \to \infty\,.
\end{aligned}
\] 
Combining the three phases, the obtained control $b = (b_t)_{t \in [0,2]}$ is predictable and square integrable, and 
\[G_2 = \int_0^2 b_s \rmd B_s = F \in \{-1,+1\}~\mbox{a.s.}\]
Therefore Theorem~\ref{thm:random_matrix_discrepancy} guarantees that $\|\cA(\widehat x)\|_{\op}\le 2\sqrt{J_\alpha(b)}+\varepsilon$ with high probability for every $\varepsilon>0$, where 
\[J_\alpha(b) = J_\alpha(b_{|[0,t_0]})+J_\alpha(b_{|[t_0,1]})+J_\alpha(b_{|[1,2]})\,.\] 
We use Lemma~\ref{lem:moment-matching} and Eq.~\eqref{eq:variance-phase-I},
and let $t_0$ be sufficiently small so that $(1+\alpha)t_0v_\eta \le 2\sqrt{\alpha}h_\eta I_\eta$ to conclude our proof. 
\end{proof}

\subsection{Discussion}

We discuss here a few interesting directions extending this approach. 

\paragraph{Analogy with the random perceptron} The vector version of this problem is the random symmetric binary perceptron problem where given a random matrix $A:\R^n \to \R^m$, $n/m \to \alpha$, of centered Gaussian entries with variance $1/m$, one is asked to find a binary point $x$ such that   
 \[\|Ax\|_{\infty} \le \sigma\,.\]
This problem and its spherical version where $x$ is only constrained to be on the sphere, have received a lot of attention, with the satisfiability threshold, structure of the solution space, and algorithmic tractability are all understood to an advanced degree; see for instance~\cite{aubin2019storage,gamarnik2022algorithms,montanari2021tractability,huang2026algorithmic} and references therein. We observe that the diffusion approach we used in this paper solves a slightly modified perceptron problem, namely if one asks for $x \in \{-1,+1\}^n$ such that 
\[\|Ax\|_2\le \sigma\,.\]           
Indeed it can be verified that an implementation of this choice of nonlinearities via classical state evolution~\cite{javanmard2013state} yields for every $\alpha>0$,
\[Am^{N} \,\xrightarrow[\delta \to 0]{\mathrm{emp}}\,N(0,J_{\alpha}(b))\,,\]   
where $n \to \infty, n/m \to \alpha$ first and $N \to \infty, \delta = T/N \to 0$ second, and  $J_\alpha(b)$ is defined in~\eqref{eq:J-alpha}. This observation, together with the analogy between Gaussians and semicircular variables, was the source of our choice of controls for matrix discrepancy.    

\paragraph{Spectral non-linearities} The iteration~\eqref{eq:AMP0}-\eqref{eq:AMP1} constructs the ``matrix messages" $M^t$ as linear combinations of the iterates $V^s, s\le t$. This was enough to obtain an operator norm strictly smaller than 2 in the matrix discrepancy application, and was technically convenient in the proof of state evolution. However this choice constrains the matrix $M^t$, and therefore the matrix output $\cA(m^t)$ to always have a semicircular strong limit. To obtain a richer family of limiting operators one might consider spectral non linearities instead:       
\[M^t =F_t(V^0,\ldots,V^t)\,,\]
where $F_t$ is suitably regular self-adjoint map. Examples would be self-adjoint polynomials, incremental maps of the form 
\[F_t(a_0,\ldots,a_t) = \sum_{s \le t-1} b_s(a_{s+1}-a_s)b_s^*\,,\qquad b_s = b_s(a_0,\ldots,a_s)\,,\]
where $b_s$ is a regular map in the formal self-adjoint variables $a_s$, and linear combinations thereof (involving different choices of $b_s, s\le t$). The emerging objects  would then be free martingales requiring free stochastic calculus tools~\cite{biane1998stochastic}, and their spectra would go beyond the semicircle. 

We also mention that we imposed a simplifying condition on the vector-side non-linearities $g_t$ in our setting, Assumption~\ref{ass:transversality}. This is a purely technical point ensuring that the empirical Gram matrices $(\ip{m^t}{m^s}_n)_{s,t}$ are almost surely invertible at every finite $n,d$. One can dispose of this assumption and ensure invertibility via a different mechanism which perturbs the messages by a small additive Gaussian noise, then take the noise variance to zero at the end by a stability argument; see~\cite{berthier2017state}. 

\paragraph{Producing a typical spectrum} A by-product of Maillard's large deviation analysis~\cite{maillard2025average} is that the spectrum of $\cA(x)$ induced by a typical (u.a.r.) point $x \in \{-1,+1\}^n$ conditional on $\|\cA(x)\|_{\op}\le 2\sigma$, $\sigma<1$ is not semicircular, but with an explicit density, see Theorem 2.2 therein. For the reasons discussed above, our algorithm only produces semicircular spectra; i.e., highly atypical points in the solution space. What are the spectra one can obtain algorithmically? In particular, can one produce a point exhibiting a typical spectrum? The above proposal might be a viable approach.        
 
\paragraph{Ground state and algorithmic thresholds} What is the asymptotic value of 
\[\min_{x \in \{-1,+1\}^n} \|\cA(x)\|_{\op}\,,\qquad \frac{2n}{d^2} \to \alpha\,?\]        
And what is the smallest operator norm achievable by Lipschitz algorithms? Based on the analogy with the random perceptron problem, the latter question might be simpler than determining the ground state value~\cite{huang2026algorithmic}.  

\vspace{0.3cm}
The remaining of this paper concerns the proof of the state evolution Theorem~\ref{thm:main-general}. In Section~\ref{sec:gaussian-conditioning} we develop a Gaussian decomposition theorem for the operator $\cA$. In Section~\ref{sec:exact-rank} we prove almost sure invertibility of certain empirical Gram matrices. In Section~\ref{sec:proof-nondeg} we execute the induction argument proving state evolution without the terminal message $\widehat m$ and the stability bound~\eqref{eq:terminal-comparison-opnorm}; see Theorem~\ref{thm:direct_SE}.  This takes as input the previous two results. In Section~\ref{sec:adapted-terminal-query} we complete the proof of Theorem~\ref{thm:main-general} addressing the terminal message.

\section{Adaptive Gaussian conditioning}
\label{sec:gaussian-conditioning}
In this section we state and prove a Gaussian decomposition result after conditioning on a sequence of linear observations. This is one of the main inputs to the proof of the state evolution Theorem~\ref{thm:main-general}.  

The following will be used repeatedly.
For $M,N\in\Symm$ we have
\begin{equation}\label{eq:cov-goe}
  \E\,\taud(A_i M)\taud(A_i N)=\frac{2}{d^2}\ip{M}{N}_d\,.
\end{equation}

Recalling that $\cA^*$ is the adjoint of $\cA$ for the normalized inner products, we have for $x \in \R^n, M\in \Symm$,
\begin{equation}
 \ip{\cA(x)}{M}_d =  \ip{x}{\cA^*(M)}_n\,.
\end{equation}
From now on we will drop the parentheses from  our notation of $\cA$ and $\cA^*$: we will write $\cA x, \cA^*M$ instead of $\cA(x),\cA^*(M)$. 

Let
\begin{equation}
H_n=(\mathbb R^n,\langle\cdot,\cdot\rangle_n)\,, \qquad\qquad K_d=(\Symm,\langle\cdot,\cdot\rangle_d)\,.
\end{equation}
Fix $T\ge 0$. Suppose that the vectors
\[
  q^0,\ldots,q^T\in H_n\,,
  \qquad
  r^0,\ldots,r^T\in K_d\,,
\]
are generated by an alternating adaptive linear observation scheme as follows. 

Let $\mathscr F_0=\sigma(u^0)$. This sigma-field is independent of $\mathcal A$. 
For each $t=0,\ldots,T$, assume that $q^t$ is
$\mathscr F_t$-measurable, define
\begin{equation}\label{eq:sigma_fieldG}
  y^t:=\mathcal A q^t\,,
  \qquad
  \mathscr G_t:=\mathscr F_t\vee \sigma(y^t)\,,
\end{equation}
assume that $r^t$ is $\mathscr G_t$-measurable, and define
\begin{equation}\label{eq:sigma_fieldF}
  x^{t}:=\mathcal A^* r^t\,,
  \qquad
  \mathscr F_{t+1}:=\mathscr G_t\vee \sigma(x^{t})\,.
\end{equation}
In our application, we will take
\[
q^t = m^t = g_t(u^0,\ldots,u^t)\,,
\qquad
r^t = M^t = \sum_{s=0}^t a_{t,s} V^{s}\,.
\]  
 It can then be easily shown by induction that $V^t$ is $\mathscr G_t$-measurable and $u^{t}$ is $\mathscr F_t$-measurable for all $0 \le t\le T$.

\begin{lemma}\label{lem:adaptive-conditioning}
Let
\[
  Q_T:\mathbb R^{T+1}\to H_n,
  \qquad
  Q_Te_t=q^t\,,~~~~~ t \le T 
\]
and
\[
  R_T:\mathbb R^{T+1}\to K_d,
  \qquad
  R_Te_t=r^t\,,~~~~~ t \le T
\]
where $(e_t)_{t=0}^{T}$ is the standard basis of $\mathbb R^{T+1}$.
Let
\[
  Y_T:=\mathcal A Q_T,
  \qquad
  X_T:=\mathcal A^*R_T\,.
\]
Assume that $Q_{T}$ and $R_T$ have full rank almost surely. Define
\[
  G_{Q}:=Q_T^*Q_T\,,
  \qquad
  G_{R}:=R_T^*R_T\,,
\]
and
\[
  P_{Q}:=Q_TG_{Q}^{-1}Q_T^*\,,
  \qquad
  P_{R}:=R_TG_{R}^{-1}R_T^*\,.
\] 
Let $\widetilde{\mathcal A}$ be a copy of $\mathcal{A}$ independent of
$\sigma(\mathcal{A}) \vee \mathscr{F}_{T+1}$. If
\[
  q=Q_T\beta+h,
  \qquad
  h\perp \operatorname{Ran}(Q_T)\,,
\]
where $q$ is $\mathscr F_{T+1}$-measurable, then, conditionally on
$\mathscr F_{T+1}$,
\begin{equation}\label{eq:adaptive-domain-vector}
  \mathcal A q
  \stackrel{\mathrm d}{=}
  Y_T\beta
  +
  R_{T}G_{R}^{-1}X_{T}^*h
  +
  P_{R}^\perp\widetilde{\mathcal A}h\,.
\end{equation}
Here and below, $R_{T-1}$ and $X_{T-1}$ denote $R_T$ and $X_T$
with their last columns deleted; when $T=0$, the corresponding terms are
absent.  If
\[
  k=R_{T-1}\gamma+\ell\,,
  \qquad
  \ell\perp \operatorname{Ran}(R_{T-1})\,,
\]
where $k$ is $\mathscr G_T$-measurable, then, conditionally on
$\mathscr G_T$,
\begin{equation}\label{eq:adaptive-range-vector}
  \mathcal A^* k
  \stackrel{\mathrm d}{=}
  X_{T-1}\gamma
  +
  Q_{T}G_{Q}^{-1}Y_{T}^*\ell
  +
  P_{Q}^\perp\widetilde{\mathcal A}^*\ell\,.
\end{equation}
\end{lemma}

\begin{proof}
For brevity write $Q=Q_T$, $R=R_T$, $Y=Y_T$, and $X=X_T$. We will first prove that conditionally on $\mathscr F_{T+1}$,
\begin{equation}\label{eq:adaptive-conditioning}
  \mathcal A
  \stackrel{\mathrm d}{=}
  Y G_Q^{-1}Q^* + R G_R^{-1}X^*P_Q^\perp + P_R^\perp\widetilde{\mathcal A}P_Q^\perp\,.
\end{equation}
This will prove \eqref{eq:adaptive-domain-vector}.  We prove the range-side formula~\eqref{eq:adaptive-range-vector} separately below by stopping the same conditioning argument at $\mathscr G_T$, before the final observation $x^T$ is revealed.

Let $\mathscr L(H_n,K_d)$ be the vector space of linear operators from $H_n$ to $K_d$ equipped with its Hilbert--Schmidt inner product: 
\[\langle B, C\rangle = \sum_{i=1}^n \ip{Bu_i}{Cu_i}_{d},\]
where $(u_i)$ is an orthonormal basis of $H_n$.
If $(F_j)$ is an orthonormal basis of $K_d$, then the coefficients
\[
  \langle \mathcal A u_i,F_j\rangle_d
\]
are independent centered Gaussian variables with common variance $2/d^2$. Indeed, we have $\mathcal A u_i \stackrel{\mathrm d}{=} A_i$, therefore each $\langle \cA u_i,F_j\rangle_d$ is a centered Gaussian. Moreover, we have by~\eqref{eq:cov-goe} that
\begin{align*}
\E\big[ \langle \cA u_i, F_j \rangle_d \langle \cA u_{i'}, F_{j'} \rangle_d \big] &= \frac1n \sum_{1 \le k, k' \le n} (u_i)_k (u_{i'})_{k'} \E\big[ \langle A_k, F_j \rangle_d \langle A_{k'}, F_{j'} \rangle_d \big] \\
&= \frac1n \one\{i=i', j=j' \} \sum_{k=1}^n \E\big[ \langle A_k, F_j \rangle_d^2 \big] = \frac2{d^2} \one\{i=i', j=j' \}\,,
\end{align*}
proving the assertion. Thus $\cA$ is a centered isotropic Gaussian element of $\mathscr L(H_n,K_d)$, up to the scalar variance $2/d^2$.

Next, we shall repeatedly use the following Gaussian fact. Let $E,F$ be finite-dimensional Hilbert spaces and let $B:E\to F$ be a centered isotropic Gaussian linear operator. If $h\in E$ is deterministic, then conditionally on $Bh$, the restriction of $B$ to $h^\perp$ is independent of $Bh$ and has an isotropic Gaussian law:
\begin{equation}\label{eq:one-step-domain}
  BP_{h^\perp}
  \stackrel{\mathrm d}{=}
  \widetilde B P_{h^\perp}\,,
\end{equation}
where $\widetilde B$ is an independent copy of $B$. Similarly, if $\ell\in F$ is deterministic, then conditionally on $B^*\ell$, 
\begin{equation}\label{eq:one-step-range}
  P_{\ell^\perp}B
  \stackrel{\mathrm d}{=}
  P_{\ell^\perp}\widetilde B\,.
\end{equation}
We now prove the adaptive conditioning statement. For $t=0,\ldots,T+1$, let
\[
  S_t:=\operatorname{span}\{q^0,\ldots,q^{t-1}\}\subset H_n\,,
  \qquad
  L_t:=\operatorname{span}\{r^0,\ldots,r^{t-1}\}\subset K_d\,,
\]
with the convention $S_0=L_0=\{0\}$. Let $P_{S_t}$ and $P_{L_t}$ denote the corresponding orthogonal projections.

We prove by induction that, conditionally on $\mathscr F_t$,
\begin{align}\label{eq:gaussianconditioning_induction_measurability}
  \mathcal A P_{S_t}
  \quad\text{and}\quad
  P_{L_t}\mathcal A
\end{align}
are $\mathscr F_t$-measurable, and the ``unrevealed block" satisfies
\begin{equation}\label{eq:residual-invariant}
  P_{L_t}^\perp \mathcal A P_{S_t}^\perp
  \stackrel{\mathrm d}{=}
  P_{L_t}^\perp \widetilde{\mathcal A} P_{S_t}^\perp\,,
\end{equation}
where $\widetilde{\mathcal A}$ is an independent copy of $\mathcal A$, independent of $\mathscr F_t$.

For $t=0$, this is immediate from the independence of $\mathscr F_0$ and $\mathcal A$, because $S_0=L_0=\{0\}$.

Now we assume~\eqref{eq:gaussianconditioning_induction_measurability},~\eqref{eq:residual-invariant} hold at time $t$, and decompose
\[
  q^t=P_{S_t}q^t+h^t\,, \qquad h^t:=P_{S_t}^\perp q^t\,.
\]
Therefore we may write
\begin{align*}
y^t=\mathcal A q^t = \mathcal A P_{S_t}q^t + P_{L_t}\mathcal A h^t + P_{L_t}^\perp \mathcal A h^t\,.
\end{align*}
Since $q^t$ and $S_t$ are $\mathscr F_t$-measurable, $h^t$ and $P_{S_t}q^t$ are $\mathscr F_t$-measurable. Thus by the inductive hypothesis, $\mathcal A P_{S_t}q^t$ and $P_{L_t}\mathcal A h^t$ are $\mathscr F_t$-measurable. Hence the only remaining random term conditionally on $\mathscr F_t$ is
\[
  P_{L_t}^\perp \mathcal A h^t
  =
  \big(P_{L_t}^\perp \mathcal A P_{S_t}^\perp\big)h^t\,.
\]

By the induction hypothesis, conditionally on $\mathscr F_t$, the operator $P_{L_t}^\perp \mathcal A P_{S_t}^\perp$ is a fresh isotropic Gaussian operator from $S_t^\perp$ to $L_t^\perp$. Applying the Gaussian conditioning fact~\eqref{eq:one-step-domain} to the vector $h^t$, we obtain that conditionally on $\mathscr F_t \vee \sigma( P_{L_t}^\perp \cA P_{S_t}^\perp h^t)= \mathscr F_t\vee\sigma(y^t) = \mathscr G_t$,
\begin{align}\label{eq:gaussianconditioning_induction_intermediatefreshgaussian}
  P_{L_t}^\perp \mathcal A P_{S_{t+1}}^\perp
  \stackrel{\mathrm d}{=}
  P_{L_t}^\perp \widetilde{\mathcal A}P_{S_{t+1}}^\perp\,,
\end{align}
with $\widetilde{\mathcal A}$ independent of $\mathscr G_t$. (The equality of the $\sigma$-algebras follows by our earlier remarks.) Moreover, $\mathcal A P_{S_{t+1}}$ is now $\mathscr G_t$-measurable, because $\mathcal A P_{S_t}$ is  $\mathscr F_t$-measurable and $y^t=\mathcal A q^t$ reveals the action of $\mathcal A$ in the new direction $q^t$. The operator $P_{L_t}\mathcal A$ is $\mathscr F_t$- and therefore $\mathscr G_t$-measurable.

Next, decompose
\[
  r^t=P_{L_t}r^t+\ell^t,
  \qquad
  \ell^t:=P_{L_t}^\perp r^t\,.
\]
Therefore
\begin{align*}
x^t = \mathcal A^*r^t = \mathcal A^*P_{L_t}r^t + P_{S_{t+1}} \cA^* \ell^t  + P_{S_{t+1}}^\perp \cA^* \ell^t\,.
\end{align*}
Recall that $r^t$ and $P_{L_t}\mathcal A$ are $\mathscr G_t$-measurable. Thus $\mathcal A^*P_{L_t}r^t$ is $\mathscr G_t$-measurable as well. In addition, the projection of $\mathcal A^*\ell^t$ onto $S_{t+1}$ is $\mathscr G_t$-measurable: for every $s\in S_{t+1}$,
\[
  \langle s,\mathcal A^*\ell^t\rangle_{n}
  =
  \langle \mathcal A s,\ell^t\rangle_{d}\,,
\]
and $\mathcal A s$, as well as $r^t, L_t$ and therefore $\ell^t$, are all $\mathscr G_t$-measurable. Therefore the only remaining random term conditionally on $\mathscr G_t$ is
\[
  P_{S_{t+1}}^\perp\mathcal A^*\ell^t
  =
  \big(P_{L_t}^\perp\mathcal A P_{S_{t+1}}^\perp\big)^*\ell^t\,.
\]

As shown by~\eqref{eq:gaussianconditioning_induction_intermediatefreshgaussian}, conditioned on $\mathscr G_t$, $P_{L_t}^\perp \widetilde{\mathcal A} P_{S_{t+1}}^\perp$ is a fresh isotropic Gaussian operator from $S_{t+1}^\perp$ to $L_t^\perp$. Applying the Gaussian conditioning fact~\eqref{eq:one-step-range} in the range direction $\ell^t$, we have that conditionally on
$\mathscr G_t\vee \sigma(\cA^* \ell^t) = \mathscr G_t\vee\sigma(x^t) = \mathscr F_{t+1}$,
\[
  P_{L_{t+1}}^\perp \mathcal A P_{S_{t+1}}^\perp
  \stackrel{\mathrm d}{=}
  P_{L_{t+1}}^\perp \widetilde{\mathcal A}P_{S_{t+1}}^\perp\,,
\]
with $\widetilde{\mathcal A}$ independent of $\mathscr F_{t+1}$.

Furthermore, $P_{L_{t+1}}\mathcal A$ is now $\mathscr F_{t+1}$-measurable. Indeed, $P_{L_t}\mathcal A$ is $\mathscr G_t$ and therefore $\mathscr F_{t+1}$-measurable, hence $\mathcal A^*P_{L_t}r^t$ is $\mathscr F_{t+1}$-measurable. Since $x^t=\mathcal A^*r^t$ together with $\mathcal A^*P_{L_t}r^t$ determines $\mathcal A^*\ell^t$, it follows that the component of $\mathcal A$ in the new range direction $\ell^t$, and hence $P_{L_{t+1}}\mathcal A$, is $\mathscr F_{t+1}$-measurable. As shown earlier, the operator $\mathcal A P_{S_{t+1}}$ is $\mathscr G_t$- and therefore $\mathscr F_{t+1}$-measurable. This completes the induction, proving~\eqref{eq:gaussianconditioning_induction_measurability},~\eqref{eq:residual-invariant}.

We now take $t=T+1$ in~\eqref{eq:residual-invariant}. Note we have
\[
  S_{T+1}=\operatorname{Ran}(Q)\,,
  \qquad
  L_{T+1}=\operatorname{Ran}(R)\,.
\]
Thus, conditionally on $\mathscr F_{T+1}$, we have
\begin{align}\label{eq:residual-invariant-result}
  P_R^\perp \mathcal A P_Q^\perp
  \stackrel{\mathrm d}{=}
  P_R^\perp\widetilde{\mathcal A}P_Q^\perp\,,
\end{align}
with $\widetilde{\mathcal A}$ independent of $\mathscr F_{T+1}$.

We now prove~\eqref{eq:adaptive-conditioning},~\eqref{eq:adaptive-domain-vector}. We first identify the deterministic, already revealed part of $\cA$. Since $Y=\mathcal A Q$,
\[
  \mathcal A P_Q
  =
  \mathcal A QG_Q^{-1}Q^*
  =
  YG_Q^{-1}Q^*\,.
\]
Similarly, since $X=\mathcal A^*R$, we have $X^*=R^*\mathcal A$, and hence
\[
  P_R\mathcal A
  =
  RG_R^{-1}R^*\mathcal A
  =
  RG_R^{-1}X^*\,.
\]
Therefore
\[
  P_R\mathcal A P_Q^\perp
  =
  RG_R^{-1}X^*P_Q^\perp\,.
\]
Using the orthogonal block decomposition
\[
  \mathcal A
  =
  \mathcal A P_Q
  +
  P_R\mathcal A P_Q^\perp
  +
  P_R^\perp\mathcal A P_Q^\perp\,,
\]
we obtain from~\eqref{eq:residual-invariant-result} that conditionally on $\mathscr F_{T+1}$,
\[
  \mathcal A
  \stackrel{\mathrm d}{=}
  YG_Q^{-1}Q^*
  +
  RG_R^{-1}X^*P_Q^\perp
  +
  P_R^\perp\widetilde{\mathcal A}P_Q^\perp\,.
\]
This proves \eqref{eq:adaptive-conditioning}.

Now let
\[
  q=Q\beta+h\,,
  \qquad
  h\perp\operatorname{Ran}(Q)\,.
\]
Then $Q^*h=0$ and $P_Q^\perp q=h$. Applying \eqref{eq:adaptive-conditioning} to $q$, we obtain
\[
\begin{aligned}
  \mathcal A q
  &\stackrel{\mathrm d}{=}
  YG_Q^{-1}Q^*(Q\beta+h)
  +
  RG_R^{-1}X^*P_Q^\perp(Q\beta+h)
  +
  P_R^\perp\widetilde{\mathcal A}P_Q^\perp(Q\beta+h) \\
  &=
  Y\beta
  +
  RG_R^{-1}X^*h
  +
  P_R^\perp\widetilde{\mathcal A}h\,.
\end{aligned}
\]
This proves \eqref{eq:adaptive-domain-vector}.

Finally, we prove~\eqref{eq:adaptive-range-vector}. Let
\[
 R_-:=R_{T-1}\,,\qquad X_-:=X_{T-1}\,,
 \qquad G_{R,-}:=R_-^*R_-\,,
 \qquad P_{R,-}:=R_-G_{R,-}^{-1}R_-^*\,.
\]
When $T=0$, all four objects are interpreted as the corresponding zero maps, and the terms containing them are absent. We stop the alternating conditioning induction after revealing $y^T=\cA q^T$, and before revealing $x^T=\cA^*r^T$.  At this point the known domain and range spaces are $\operatorname{Ran}(Q_T)$ and $\operatorname{Ran}(R_-)$, respectively. The same calculation used for \eqref{eq:adaptive-conditioning} gives, conditionally on $\mathscr G_T$,
\begin{equation}\label{eq:adaptive-conditioning-range-stop}
 \cA
 \stackrel{\mathrm d}{=}
 YG_Q^{-1}Q^*
 +R_-G_{R,-}^{-1}X_-^*P_Q^\perp
 +P_{R,-}^\perp\widetilde{\cA}P_Q^\perp\,,
\end{equation}
where $\widetilde{\cA}$ is independent of $\mathscr G_T$.  In our current notation, we may write
\[
  k=R_-\gamma+\ell\,, \qquad \ell\perp\operatorname{Ran}(R_-)\,.
\]
Taking adjoints in~\eqref{eq:adaptive-conditioning-range-stop} yields
\[
  \mathcal A^*
  \stackrel{\mathrm d}{=}
  QG_Q^{-1}Y^*
  +P_Q^\perp X_-G_{R,-}^{-1}R_-^*
  +P_Q^\perp\widetilde{\mathcal A}^*P_{R,-}^\perp\,.
\]
Therefore
\[
\begin{aligned}
  \mathcal A^*k
  &\stackrel{\mathrm d}{=}
  QG_Q^{-1}Y^*(R_-\gamma+\ell)
  +
  P_Q^\perp X_-G_{R,-}^{-1}R_-^*(R_-\gamma+\ell)
  +
  P_Q^\perp\widetilde{\mathcal A}^*P_{R,-}^\perp(R_-\gamma+\ell)  \\
  &=
  QG_Q^{-1}Y^*R_-\gamma
  +
  QG_Q^{-1}Y^*\ell
  +
  P_Q^\perp X_-\gamma
  +
  P_Q^\perp\widetilde{\mathcal A}^*\ell\,.
\end{aligned}
\]
Now observe that
\[
  Y^*R_-=Q^*\mathcal A^*R_-=Q^*X_-\,.
\]
This implies
\[
  QG_Q^{-1}Y^*R_-\gamma
  =
  QG_Q^{-1}Q^*X_-\gamma
  =
  P_QX_-\gamma\,.
\]
Combining with the earlier display thus gives
\begin{align}
\mathcal A^*k
  &\stackrel{\mathrm d}{=}
  P_QX_-\gamma
  +
  P_Q^\perp X_-\gamma
  +
  QG_Q^{-1}Y^*\ell
  +
  P_Q^\perp\widetilde{\mathcal A}^*\ell \notag \\
  &= X_-\gamma
  +
  QG_Q^{-1}Y^*\ell
  +
  P_Q^\perp\widetilde{\mathcal A}^*\ell\,. \notag
\end{align}
This proves \eqref{eq:adaptive-range-vector}, completing the proof of Lemma \ref{lem:adaptive-conditioning}.
\end{proof}

\begin{lemma}\label{lem:projected-innovation}
Let $Q:\R^p\to H_n$ and $R:\R^q\to K_d$ be random linear
operators where $p,q$ are fixed, and let $P_Q,P_R$ be the
orthogonal projections onto their ranges, and $h \in H_n$ and $k \in K_d$ be random elements. Let $\widetilde\cA$ be independent of $(Q,R,h,k)$ and have the same law as $\cA$. Then conditionally on $(Q,R,h,k)$,
\[
 \|P_R\widetilde\cA h\|_{\op}
 =O_{\Prob}\big(d^{-1/2}\|h\|_n\big)\,,
 \qquad
 \|P_Q\widetilde\cA^*k\|_n
 =O_{\Prob}\big(d^{-1}\|k\|_{2,d}\big)\,.
\]
Moreover, conditionally on $h$, $\widetilde\cA h$ has the law
$\|h\|_nW_d$, where $W_d$ is standard GOE. Conditionally on $k$,
the coordinates of $\widetilde\cA^*k$ are independent centered
Gaussians with variance $\alpha_d\|k\|_{2,d}^2$.
\end{lemma}

\begin{proof}
Condition on $(Q,R,h,k)$. Let $B_1,\ldots,B_r$ be an
$\langle\cdot,\cdot\rangle_d$-orthonormal basis of
$\operatorname{Ran}(R)$, where $r\le q$. Then
\[
 P_R\widetilde\cA h
 =\sum_{j=1}^r
 \langle B_j,\widetilde\cA h\rangle_d B_j\,.
\]
By \eqref{eq:cov-goe}, each coefficient is centered Gaussian with variance $2\|h\|_n^2/d^2$. Indeed, we have
\begin{align*}
\langle B_j, \widetilde \cA h \rangle_d = \frac1{\sqrt{n}} \sum_{i=1}^n h_i \langle B_j, \widetilde A_i \rangle_d \,,\qquad\qquad \Var\big( \langle B_j, \widetilde A_i \rangle_d \big) = \frac2{d^2} \langle B_j, B_j \rangle_d = \frac2{d^2} \,,
\end{align*}
where the second equality uses~\eqref{eq:cov-goe}. Combining these displays yields $\Var \langle B_j,\widetilde\cA h\rangle_d = 2\|h\|_n^2/d^2$. Moreover, $\|B_j\|_{\op}\le\sqrt d\|B_j\|_{2,d}=\sqrt d$. Since $r$ is fixed, we obtain
\[
 \|P_R\widetilde\cA h\|_{\op}
 \le\sum_{j=1}^r
 \big|\langle B_j,\widetilde\cA h\rangle_d \big| \cdot \|B_j\|_{\op}
 =O_{\Prob}\big(d^{-1/2}\|h\|_n \big)\,.
\]

Likewise, let $b_1,\ldots,b_s$ be an
$\langle\cdot,\cdot\rangle_n$-orthonormal basis of
$\operatorname{Ran}(Q)$, where $s\le p$. By properties of the adjoint,
\[
 P_Q\widetilde\cA^*k
 =\sum_{j=1}^s\langle\widetilde\cA b_j,k\rangle_d b_j\,.
\]
By the same rationale as above, using~\eqref{eq:cov-goe}, we obtain that each coefficient is centered Gaussian with variance
$2\|k\|_{2,d}^2/d^2$. Thus as $s$ is fixed,
\[
 \|P_Q\widetilde\cA^*k\|_n^2
 =\sum_{j=1}^s \big|\langle\widetilde\cA b_j,k\rangle_d \big|^2
 =O_{\Prob}\big(d^{-2}\|k\|_{2,d}^2 \big)\,.
\]

The distributional assertions follow from the same rationale as the variance calculations above, using~\eqref{eq:cov-goe}, and as $\widetilde \cA$ is independent of $h$ and $k$.
\end{proof}

\section{Gram matrix invertibility}
\label{sec:exact-rank}
The conditioning proof uses inverses of empirical and limiting Gram
matrices. In this section we prove their almost-sure invertibility.  

For any $t\ge 1$ let
\[
  S_t:=\Span\{m^0,\ldots,m^{t-1}\}\,,
 \qquad
  L_t:=\Span\{M^0,\ldots,M^{t-1}\}\,,
\]
 and recall the sigma-fields $\mathscr F_t, \mathscr G_t$ produced by the alternating observation scheme \eqref{eq:sigma_fieldG}--\eqref{eq:sigma_fieldF}.
The proof of Lemma~\ref{lem:adaptive-conditioning} gives the following statement.

\begin{lemma}
\label{lem:rank-free-block}
For any $t$, conditionally on the
revealed sigma-field $\mathscr F_{t-1}$,
\begin{equation}\label{eq:subspace-fresh-block}
 P_{\mathcal S_t}^{\perp}\cA P_{\mathcal S_t}^{\perp}
 \stackrel{\mathrm d}{=}
 P_{L_t}^{\perp}\widetilde\cA P_{S_t}^{\perp}\,,
\end{equation}
where $\widetilde\cA$ is an independent copy of $\cA$.  
\end{lemma}

\begin{proposition}
\label{prop:exact-rank}
Assume \ref{ass:init}, \ref{ass:a1}, and
\ref{ass:transversality}.  If
$n\ge T+1$ and $D_d = \dim K_d \ge T+1$, then almost surely for every
$0\le t\le T$,
\begin{equation}\label{eq:exact-rank}
 \rank[m^0|\cdots|m^t]
 =\rank[M^0|\cdots|M^t]=t+1\,.
\end{equation}
\end{proposition}

\begin{proof}
At time $t=0$, $m^0 \neq 0$ by Assumption~\ref{ass:transversality}, and
$M^0=a_{0,0}\cA(m^0)\ne0$ almost surely.

Suppose $1\le t\le T$ and the histories preceding time $t$ have full
rank.  Conditionally on the ``half-step" $\mathscr G_{t-1}$, Lemma~\ref{lem:rank-free-block}
gives
\begin{equation}\label{eq:exact-rank-vector-innovation}
 u^t\stackrel{\mathrm d}{=}
 \mu^t+P_{S_{t}}^{\perp}\widetilde\cA^*K^{t-1},
 \qquad
 K^{t-1}:=P_{L_{t-1}}^{\perp}M^{t-1}\ne0,
\end{equation}
where $\mu^t$ is $\mathscr G_{t-1}$-measurable and $\widetilde\cA$ is independent
of $\mathscr G_{t-1}$.  The second term is centered Gaussian
on $ S_t^\perp$, with conditional covariance
\[
 \alpha_d\|K^{t-1}\|_{2,d}^2P_{ S_t}^{\perp}\,.
\]
It is therefore nondegenerate on that space.  By
\eqref{eq:transverse-form} we have
\[
 m^t=r_t(u^0,\ldots,u^{t-1})+D_tu^t\,,
 \qquad
 D_t:=\diag\bigl(c_t(u_i^0,\ldots,u_i^{t-1})\bigr)_{i\le n}\succ 0\,.
\]
in the coordinatewise sense. Let $h^t=P_{\mathcal S_t}^\perp m^t$.
Since $D_t$ is invertible, $h^t$ has a conditional density on $S_t^\perp$.  It is therefore nonzero almost surely.
Moreover, triangularity and
$a_{t,t}\ne0$ imply
\[
 \Span\{M^0,\ldots,M^{t-1}\}
 =\Span\{V^0,\ldots,V^{t-1}\}\,,
 \qquad \cA(S_t)\subseteq L_t\,.
\]
The matrix Onsager correction $\ons_V^t$ also belongs to $\mathcal L_t$,
hence
\begin{equation}\label{eq:exact-rank-matrix-innovation}
 P_{L_t}^\perp M^t
 =a_{t,t}P_{L_t}^\perp\cA h^t.
\end{equation}
By Lemma~\ref{lem:rank-free-block}, the right-hand side is conditionally
a nondegenerate Gaussian in $\mathcal L_t^\perp$, so it is nonzero
almost surely.  This closes the induction.  
\end{proof}

\begin{lemma}
\label{lem:limiting-Gram-positivity}
For the state evolution under Assumptions~\ref{ass:init},
\ref{ass:a1}, \ref{ass:g}, and \ref{ass:transversality}, for every
$t\le T$,
\[
 \Gamma_{G,t}:=(\E[G^rG^s])_{0\le r,s\le t}\succ0\,,
 \qquad
 \Gamma_{W,t}:=(\taup(W^rW^s))_{0\le r,s\le t}\succ0\,.
\]
\end{lemma}

\begin{proof}
The first Gram matrix is positive at time zero because $G^0\ne0$ almost
surely.  If $\Gamma_{G,t-1}\succ0$, then
\[
 \Gamma_{W,t-1}=A_{t-1}\Gamma_{G,t-1}A_{t-1}^{\top}\succ0
\]
where $A_{t-1}=(a_{r,s})_{0\le s\le r<t}$.  This matrix is triangular
with nonzero diagonal.  The covariance identity~\eqref{eq:covU} implies that  
the Gaussian vector $(U^1,\ldots,U^t)$ has positive-definite
covariance, and
\[
 v_t:=\Var(U^t\mid U^1,\ldots,U^{t-1})>0\,.
\]
For every $b\in\R^t$, the law of total variance implies
\[\E\left[\left(G^t-\sum_{s<t}b_sG^s\right)^2\right]
 \ge \E[\Var(G^t\mid U^0,\ldots,U^{t-1})]\,,\]
while from \eqref{eq:transverse-form} we have
\begin{equation*}
 G^t =r_t(U^0,\ldots,U^{t-1})
  +c_t(U^0,\ldots,U^{t-1})U^t\,,
 \qquad c_t>0\,.
\end{equation*}
Therefore the above variance is no less than
\[v_t\E[c_t(U^0,\ldots,U^{t-1})^2]>0\,.\]
This proves positivity of the next Schur complement of $\Gamma_{G,t}$ (by taking $b = \Gamma_{G,t-1}^{-1}(\E[G^sG^t])_{s<t}$).
Induction proves both claims.
\end{proof}

\section{Proof of state evolution}
\label{sec:proof-nondeg}
This is the section where we prove ``one half" of Theorem~\ref{thm:main-general} stated below: 
\begin{theorem}
\label{thm:direct_SE}
The following holds under Assumptions \ref{ass:aspect}, \ref{ass:init},
\ref{ass:a1}, \ref{ass:g}, and \ref{ass:transversality}.  For fixed $T \ge 1$, let $(u^{t+1},m^{t},V^{t},M^{t})_{t \le T}$ be generated by
the AMP iteration \eqref{eq:AMP0}--\eqref{eq:AMP1}.  Then, 
\begin{align}
  \bigl(u_i^0,\ldots,u_i^{T+1},m_i^0,\ldots,m_i^{T}\bigr)_{i\le n}
  &\xrightarrow{\mathrm{emp}}
  \bigl(U^0,\ldots,U^{T+1},G^0,\ldots,G^T\bigr)\,,\label{eq:vector-SE-main1}\\
\mbox{and}\qquad
  (V^0,\ldots,V^T,M^0,\ldots,M^T)
  &\xrightarrow{\mathrm{str}}
  (S^0,\ldots,S^T,W^0,\ldots,W^T)\,,\label{eq:matrix-SE-main1}
\end{align}
where the limiting process
$(U^0,\ldots,U^{T+1},S^0,\ldots,S^T)$ has covariances given by
\eqref{eq:covU}--\eqref{eq:covS}.
\end{theorem}
The terminal step $\widehat m$ together with the stability of $\cA(\widehat m - m)$ will be proved separately in Section~\ref{sec:adapted-terminal-query}.

We will use the strong convergence of independent GOE matrices jointly with an already strongly converging array of matrices.   
The statement used is Theorem 4.3 in Fan, Sun and Wang~\cite{fan2021principal} which we reproduce here:

\begin{proposition}[\cite{fan2021principal}]\label{prop:strong-free}
Let $W_1^d,\ldots,W_k^d$ be independent standard GOE matrices with diagonal and off-diagonal variances $2/d$ and $1/d$ respectively, independent of deterministic or random matrices $B_1^d,\ldots,B_\ell^d$. Suppose strong convergence of the latter array in noncommutative distribution: 
\[(B_1^d,\ldots,B_\ell^d) \xto{\mathrm{str}} (b_1,\ldots,b_\ell)\,.\]   
Then we have the joint strong convergence
\[
  (W_1^d,\ldots,W_k^d,B_1^d,\ldots,B_\ell^d)
\xrightarrow{\mathrm{str}} 
  (s_1,\ldots,s_k,b_1,\ldots,b_\ell)\,,
\]
where $s_1,\ldots,s_k$ are freely independent standard semicircular variables, free from $(b_1,\ldots,b_\ell)$.
\end{proposition}

The GUE version of this result was proved earlier by Male~\cite{male2012norm}, and it upgrades the breakthrough result of strong convergence to GUE matrices by Haagerup--Thorbjornsen~\cite{haagerup2005new}. Schultz~\cite{schultz2005non} then proved the GOE version of the result in~\cite{haagerup2005new}.

\subsection{The induction argument}
 We prove Theorem~\ref{thm:direct_SE} by induction on time, following Bolthausen's conditioning argument~\cite{bolthausen2014iterative}.
We spell out our induction hypothesis:  
\begin{assumption}[Induction hypothesis $\mathcal{H}_t$]\label{ass:induction-hypotheses}
At time $t \ge 0$ assume
\begin{enumerate}[label=(\roman*)]
\item empirical convergence with all polynomial moments of the
coordinate array:
\[
  (u_i^0,\ldots,u_i^t,m_i^0,\ldots,m_i^t)_{i\le n}
\xto{\mathrm{emp}}
(U^0,\ldots,U^t,G^0,\ldots,G^t)\,;
\]
\item strong convergence of the matrix array in noncommutative distribution: 
\[
(V^0,\ldots,V^{t-1},M^0,\ldots,M^{t-1}) 
\xto{\mathrm{str}}
(S^0,\ldots,S^{t-1},W^0,\ldots,W^{t-1})\,.
\]
\end{enumerate}
These assumptions will be collectively denoted by $\mathcal{H}_t$.
\end{assumption}


Let us now record a few lemmas for future reference. 

\begin{lemma}
\label{lem:emp-calc}
Let $X_1^d,\ldots,X_n^d\in\R^q$ and suppose
\[
  (X_i^d)_{i\le n}\xto{\mathrm{emp}}X\,,
\]
Let
$f:\R^k\to\R$ be such that $f$ and the functions
$\partial_1 f,\ldots,\partial_{k}f$ are pseudo-Lipschitz of finite
order.  Then
\[
 (X_i^d,f(X_i^d))_{i\le n}
 \xrightarrow{\mathrm{emp}}(X,f(X))
\]
and, for every $1\le j\le k$,
\[
 \frac1n\sum_{i=1}^n\partial_j f(X_i^d)
 \xrightarrow{p}\E[\partial_j f(X)]\,.
\]
In particular, all empirical polynomial moments of $(X_i^d,f(X_i^d))_{i \le n}$ are $O_{\Prob}(1)$.
\end{lemma}

\begin{proof}
A pseudo-Lipschitz function of order $r$ has growth at most
$C(1+\|x\|^r)$.  Hence, if $\Psi$ is pseudo-Lipschitz of order
$s$, the composition
\[
 x\longmapsto\Psi(x,f(x))
\]
is pseudo-Lipschitz of some finite order depending only on $r$ and
$s$.  The first assertion therefore follows directly from empirical
convergence.  The second follows by using $\partial_jf$ as a
test function.  Applying the first assertion to even powers of each
coordinate gives the moment statement.
\end{proof}

\begin{lemma}
\label{lem:conditional-empirical}
Let $\mathscr H_d$ be a sigma-field, and let
$X_1^d,\ldots,X_n^d\in\R^q$ be $\mathscr H_d$-measurable.  Suppose
\[
  (X_i^d)_{i\le n}\xto{\mathrm{emp}}X\,,
\]
where $X$ has moments of all orders.  Let $b_d\in\R^q$ and
$\sigma_d\ge0$ be $\mathscr H_d$-measurable and assume
\[
  b_d\xto{p}b,\qquad \sigma_d\xto{p}\sigma
\]
for deterministic $b\in\R^q$ and $\sigma\ge0$.  Let
$Z_1,\ldots,Z_n$ be independent standard Gaussians, independent of
$\mathscr H_d$, and let $r^d\in\R^n$ satisfy
\begin{equation}\label{eq:conditional-empirical-error}
  \big\|r^d\big\|_n\xto{p}0,
  \qquad
  \frac1n\sum_{i=1}^n\abs{r_i^d}^{2p}=O_{\Prob}(1)
  \quad\text{for every fixed }p\ge1\,.
\end{equation}
Then, with
\[
  Y_i^d=b_d^\top X_i^d+\sigma_dZ_i+r_i^d\,,
\]
one has
\[
  (X_i^d,Y_i^d)_{i\le n}
  \xto{\mathrm{emp}}
  (X,b^\top X+\sigma Z),
\]
where $Z\sim N(0,1)$ is independent of $X$.
\end{lemma}

\begin{proof}
Fix $k\ge1$ and $\Psi\in\PL_k(\R^{q+1})$ and let
$Y_{i,0}^d=b_d^\top X_i^d+\sigma_dZ_i$.  The pseudo-Lipschitz inequality
and Cauchy--Schwarz give
\[
\begin{aligned}
 &\frac1n\sum_{i=1}^n
 \abs{\Psi(X_i^d,Y_i^d)-\Psi(X_i^d,Y_{i,0}^d)}\\
 &\quad\le C_\Psi
 \left\{\frac1n\sum_{i=1}^n
 \left(1+\norm{X_i^d}^{k-1}+\abs{Y_{i,0}^d}^{k-1}
       +\abs{r_i^d}^{k-1}\right)^2\right\}^{1/2}
 \big\|r^d\big\|_n=o_{\Prob}(1)\,.
\end{aligned}
\]
Conditionally on $\mathscr H_d$, the variables
$\Psi(X_i^d,Y_{i,0}^d)$ are independent.  Their conditional empirical
variance is bounded by
\[
  \frac{C}{n}\left(1+\frac1n\sum_i\norm{X_i^d}^{2k}
      +\norm{b_d}_2^{2k}+\sigma_d^{2k}\right)=o_{\Prob}(1)\,.
\]
Thus the empirical average may be replaced by
$\frac1n\sum_i\overline\Psi_d(X_i^d)$, with 
$\overline\Psi_d(x) :=\E_Z\Psi(x,b_d^\top x+\sigma_dZ)$, 
which can in turn be replaced by 
$\frac1n\sum_i\overline\Psi(X_i^d)$, with $\overline\Psi(x) :=\E_Z\Psi(x,b^\top x+\sigma Z)$ 
via the same pseudo-Lipschitz control, leveraging $b_d \to b$ and $\sigma_d \to \sigma$. Empirical convergence of $X_i^d$ therefore yields the claim.
\end{proof}

Now we dive into the induction argument. 
The base step is straightforward:
By Assumption~\ref{ass:init} and Lemma~\ref{lem:emp-calc},
\[
 (u_i^0,m_i^0)_{i\le n}
 \xrightarrow{\mathrm{emp}}(U^0,G^0)\,,
 \qquad G^0=g_0(U^0)\,.
\]
This proves $\mathcal H_0$. 
%
%

\subsection{The matrix update}
\label{sec:matrix-update}
We now assume $\mathcal{H}_{t}$, $t\ge 1$. 
We prove the matrix step by writing the conditional law of $\cA(m^t)$ using Lemma~\ref{lem:adaptive-conditioning}. 
Define the operators
\[
  Q_{t-1}c:=\sum_{k=0}^{t-1}c_km^k,
  \qquad
  R_{t-1}b:=\sum_{r=0}^{t-1}b_rM^r,
\]
and their Gram matrices
\[
  G_{m,t-1}:=Q_{t-1}^*Q_{t-1}\,,
  \qquad
  G_{M,t-1}:=R_{t-1}^*R_{t-1}\,,
\]
and let
\begin{equation}
      Y_{t-1}:=\cA Q_{t-1},
  \qquad
  X_{t-1}:=\cA^* R_{t-1}.
\end{equation}

By $\mathcal{H}_{t}$ and Lemma~\ref{lem:limiting-Gram-positivity},
the empirical Gram matrices
converge to positive-definite deterministic limits. By Proposition~\ref{prop:exact-rank} their
 inverses are defined almost surely, and the continuous mapping
theorem gives
\[
  \|G_{m,t-1}^{-1}\|_{\op}
  +\|G_{M,t-1}^{-1}\|_{\op}
  \le O_{\Prob}(1)\,.
\]  
Let
\begin{equation}\label{eq:beta-h-def}
  \beta_t:=G_{m,t-1}^{-1}Q_{t-1}^*m^t\,,
  \qquad
  h^t:=m^t-Q_{t-1}\beta_t\,.
\end{equation}
Then $h^t\perp\operatorname{Ran}(Q_{t-1})$.  For $0\le k\le t$, define a vector $d_k^{(t)}\in\R^{t}$ by
\begin{equation}\label{eq:derivative-vectors}
  (d_k^{(t)})_r
  :=d_{k,r+1}\ind_{\{r<k\}} = \ind_{\{r<k\}}
  \left\langle\partial_{r+1}g_k(u^0,\ldots,u^k)\right\rangle_n,
  \qquad 0\le r<t\,.
\end{equation}
  In this notation,
\begin{equation}\label{eq:onsv-matrix-notation}
  \ons_V^k=\alpha_dR_{t-1}d_k^{(t)},
  \qquad k\le t\,,
\end{equation}
as can be verified from the definition of $\ons_V^k$ in~\eqref{eq:onsdef}. 
Also define
\begin{equation}\label{eq:delta-t-def}
  \delta_t:=d_t^{(t)}-\sum_{k<t}(\beta_t)_k d_k^{(t)}\,.
\end{equation}

\begin{lemma}\label{lem:matrix-side-regression}
Under the induction hypothesis $\mathcal{H}_{t}$,
\begin{equation}\label{eq:matrix-side-regression}
  \norm{X_{t-1}^*h^t-\alpha_dG_{M,t-1}\delta_t}_{\R^t}
  \xto{p} 0 \,.
\end{equation}
Consequently,
\begin{equation}\label{eq:matrix-side-regression-after-inverse}
  R_{t-1}G_{M,t-1}^{-1}X_{t-1}^*h^t
  =\alpha_dR_{t-1}\delta_t+o_{\op}(1)\,.
\end{equation}
\end{lemma}

\begin{proof}
From the definition of $X_{t-1}$ we have for $1\le r<t$, 
\begin{equation}
  (X_{t-1}^*h^t)_r = \ip{X_{t-1}^*h^t}{e_r}_{\R^t} = \ip{h^t}{X_{t-1}e_r}_{n} 
  =\ip{h^t}{\cA^*M^r}_n\,.
\end{equation}
Since $\cA^*M^r=u^{r+1}+\ons_u^r$, and by \eqref{eq:onsdef}, $\ons_u^r\in\Span\{m^0,\ldots,m^r\}$, while $h^t$ is orthogonal to $\Span\{m^0,\ldots,m^{t-1}\}$, we have
\begin{equation}\label{eq:Xh-equals-uh}
(X_{t-1}^*h^t)_r =\ip{u^{r+1}}{h^t}_n\,.
\end{equation}
By $\mathcal{H}_{t}$ and Gram matrix invertibility, the projection coefficients $\beta_t$ converge to the coefficients $\beta_t^\infty$ of the $L^2$-projection of $G^t$ onto $\Span\{G^0,\ldots,G^{t-1}\}$. Therefore the right side of \eqref{eq:Xh-equals-uh} converges to
\begin{equation}\label{eq:limit-Xh}
  \E\left[U^{r+1}
  \left(G^t-\sum_{k<t}(\beta_t^\infty)_kG^k\right)\right]\,.
\end{equation}
 The vector $(U^1,\ldots,U^t)$ is centered Gaussian and independent of $U^0$. The additive variable $\eps Z^k$ in $G^k$ is independent of all
$U$-fields, and hence contributes zero to $\E[U^{r+1}G^k]$. Multivariate Gaussian integration by parts and the covariance recursion~\eqref{eq:covU} gives, for $k\le t$,
\begin{align}
  \E[U^{r+1}G^k]
  &=\sum_{s=0}^{k-1}
    \E[U^{r+1}U^{s+1}]
    \E\bigl[\partial_{s+1}g_k(U^0,\ldots,U^k)\bigr] \notag\\
  &=\alpha\sum_{s=0}^{k-1}
    \taup(W^r W^s)
    \E\bigl[\partial_{s+1}g_k(U^0,\ldots,U^k)\bigr]\,.
  \label{eq:limiting-Stein}
\end{align}
Using \eqref{eq:limiting-Stein} in \eqref{eq:limit-Xh}, we find that 
\[
  X_{t-1}^*h^t \,\xto{p} \,\alpha\, G_{W,t-1}
  \left(d_t^\infty-\sum_{k<t}(\beta_t^\infty)_kd_k^\infty\right)\,,
\]
where $G_{W,t-1}=(\taup(W^rW^s))_{0\le r,s<t}$ and the vectors $d_k^\infty$ are the limits of \eqref{eq:derivative-vectors}.  On the other hand, $\mathcal{H}_t$ and Lemma~\ref{lem:emp-calc} imply
\[
  G_{M,t-1}\xto{p}G_{W,t-1},
  \qquad
  \delta_t\xto{p}
  d_t^\infty-\sum_{k<t}(\beta_t^\infty)_kd_k^\infty\,,
  \qquad
  \alpha_d\to\alpha\,.
\]
This proves \eqref{eq:matrix-side-regression}.  Since the inverse Gram matrix is bounded and the columns of $R_{t-1}$ are uniformly bounded in operator norm, \eqref{eq:matrix-side-regression-after-inverse} follows.
\end{proof}

Consider the sigma-field 
\begin{equation}\label{eq:past-minus}
  \mathcal F_t^-:= \sigma\bigl(u^0,R_{t-1},X_{t-1},Q_{t-1},Y_{t-1}\bigr) = \mathscr F_{t}\,,
\end{equation}
as per the notation of~\eqref{eq:sigma_fieldF}.

We now apply Lemma \ref{lem:adaptive-conditioning} to the decomposition \eqref{eq:beta-h-def}.  Conditionally on $\mathcal F_t^-$,
\begin{equation}\label{eq:conditional-Am-t}
  \cA m^t
  \overset{\mathrm d}{=}
  Y_{t-1}\beta_t
  +R_{t-1}G_{M,t-1}^{-1}X_{t-1}^*h^t
  +P_{R_{t-1}}^\perp\widetilde{\cA} h^t.
\end{equation}
We combine the first two terms before the Onsager subtraction.  By the definition of $Y_{t-1}$ and \eqref{eq:onsv-matrix-notation},
\begin{align}
  Y_{t-1}\beta_t
  &=\sum_{k<t}(\beta_t)_k\cA (m^k) \notag\\
  &=\sum_{k<t}(\beta_t)_kV^k
    +\alpha_dR_{t-1}\sum_{k<t}(\beta_t)_k d_k^{(t)}\,.
  \label{eq:first-conditional-mean}
\end{align}
By Lemma \ref{lem:matrix-side-regression}, the second conditional-mean term in \eqref{eq:conditional-Am-t} is
\begin{equation}\label{eq:second-conditional-mean}
  \alpha_dR_{t-1}
  \left(d_t^{(t)}-\sum_{k<t}(\beta_t)_kd_k^{(t)}\right)
  +o_{\op}(1)\,.
\end{equation}
Adding \eqref{eq:first-conditional-mean} and \eqref{eq:second-conditional-mean} yields the sought-after cancellation:
\begin{equation}\label{eq:combined-conditional-mean-matrix}
  Y_{t-1}\beta_t
  +R_{t-1}G_{M,t-1}^{-1}X_{t-1}^*h^t
  =\sum_{k<t}(\beta_t)_kV^k
   +\ons_V^t+o_{\op}(1)\,.
\end{equation}
Subtracting $\ons_V^t$ from \eqref{eq:conditional-Am-t}, and using Lemma \ref{lem:projected-innovation} we obtain that conditionally on $\mathcal F_t^-$,
\begin{equation}\label{eq:matrix-innovation-decomposition}
  V^t
  \stackrel{d}{=}\sum_{k<t}(\beta_t)_kV^k
   +\norm{h^t}_nW_t^d
   +o_{\op}(1)\,,
\end{equation}
where $W_t^d$ is a standard GOE matrix independent of the past.  By empirical convergence,
\begin{equation}\label{eq:h-var-limit-detailed}
  \norm{h^t}_n^2 ~\xto{p} ~
  q_{V,t} := 
  \dist_{L^2(\Omega,\Prob)}\left(G^t,\Span\{G^0,\ldots,G^{t-1}\}\right)^2\,,
\end{equation}
the squared distance of $G^t$ from $\Span\{G^s: s< t\}$ in $L^2(\Omega,\Prob)$, and $\beta_t\to\beta_t^\infty$.  Hence Proposition \ref{prop:strong-free} applied to \eqref{eq:matrix-innovation-decomposition} gives the strong joint limit
\[
  S^t=\sum_{k<t}(\beta_t^\infty)_kS^k+\sqrt{q_{V,t}}\,s_t\,,
\]
where $s_t$ is a standard semicircular variable free from the past. The regression coefficients $(\beta_t^\infty)_k$ and the residual variance $q_{V,t}$ are those of a centered semicircular family  $(S^0,\ldots,S^t)$ with covariance $\taup(S^tS^r)=\E[G^tG^r]$: this can be seen from the fact that free cumulants of order 2 are multilinear~\cite{mingo2017free}, or by considering a correlated family of GOE matrices and passing to the large-size limit.      
Finally, since
\[
  M^t=\sum_{s=0}^t a_{t,s} V^{s}
\]
is a fixed linear combination of $(V^0,\ldots,V^t)$,  joint strong convergence holds:
\[
  (V^0,\ldots,V^t,M^0,\ldots,M^t)
  \xto{\mathrm{str}}
  (S^0,\ldots,S^t,W^0,\ldots,W^t)\,.
\]

\subsection{The vector update}\label{sec:vector-update-detailed}

We next prove the vector step under $\mathcal{H}_{t}$.  Recall
\[
  Q_tc:=\sum_{r=0}^tc_rm^r\,,
  \qquad
  R_{t-1}b:=\sum_{s=0}^{t-1}b_sM^s\,,
\]
and
\[
  G_{m,t}:=Q_t^*Q_t\,,
  \qquad
  G_{M,t-1}:=R_{t-1}^*R_{t-1}\,,
  \qquad
  Y_t:=\cA Q_t\,,
  \qquad
  X_{t-1}:=\cA ^*R_{t-1}\,,
\]
and
\[
 \|G_{m,t}^{-1}\|_{\op}
 +\|G_{M,t-1}^{-1}\|_{\op}=O_{\Prob}(1)\,.
\]
The relevant past sigma-field is
\begin{equation}\label{eq:past-plus}
  \mathcal F_t^+:= \sigma\bigl(u^0,R_{t-1},X_{t-1},Q_t,Y_t\bigr) = \mathscr{G}_{t}\,, %
\end{equation}
as per the notation of~\eqref{eq:sigma_fieldG}.

Decompose the new matrix message relative to its predecessors:
\begin{equation}\label{eq:gamma-k-def}
  \gamma_t:=G_{M,t-1}^{-1}R_{t-1}^*M^t\,,
  \qquad
  K^t:=M^t-R_{t-1}\gamma_t\,.
\end{equation}
Then $K^t\perp\operatorname{Ran}(R_{t-1})$.  For $0\le s\le t$, let
\[
   a_s^{(t)}
  :=(a_{s,0},\ldots,a_{s,s},0,\ldots,0)\in\R^{t+1}
\]
be the coefficient vector of $M^s$ in $V^0,\ldots,V^t$, and define
\begin{equation}\label{eq:lambda-innovation}
  \widetilde{a}_t
  := a_t^{(t)}
  -\sum_{0\le s<t}(\gamma_t)_s a_s^{(t)}\,.
\end{equation}
Then we have the exact identity
\begin{equation}\label{eq:K-as-V-combination}
  K^t=\sum_{r=0}^t
  (\widetilde{a}_t)_rV^r\,.
\end{equation}

We apply Lemma \ref{lem:adaptive-conditioning} to the decomposition \eqref{eq:gamma-k-def} gives, conditionally on $\mathcal F_t^+$,
\begin{equation}\label{eq:conditional-Astar-Mt}
  \cA^* M^t
  \overset{\mathrm d}{=}
  X_{t-1}\gamma_t
  +Q_tG_{m,t}^{-1}Y_t^*K^t
  +P_{Q_t}^\perp\widetilde{\cA}^*K^t\,.
\end{equation}
We now identify the second conditional-mean term.  The $r$-th column of $Y_t$ is
\[
  \cA m^r=V^r+\ons_V^r\,.
\]
Since $\ons_V^r\in\Span\{M^0,\ldots,M^{r-1}\}\subseteq\operatorname{Ran}(R_{t-1})$ and $K^t\perp\operatorname{Ran}(R_{t-1})$, we obtain
\begin{equation}\label{eq:YstarK-exact}
  (Y_t^*K^t)_r
  =\ip{V^r}{K^t}_d\,,
\end{equation}
similarly to the argument presented for the matrix side. 
Combining \eqref{eq:YstarK-exact} with \eqref{eq:K-as-V-combination},
\begin{equation}\label{eq:YstarK-GramV}
  Y_t^*K^t
  =G_{V,t}\widetilde{a}_t\,,
  \qquad
  G_{V,t}:=\bigl(\ip{V^r}{V^s}_d\bigr)_{0\le r,s\le t}\,.
\end{equation}
By the just-proved matrix step and assertion $(i)$ of $\mathcal{H}_t$ on empirical convergence,
\[
  G_{V,t}\xto{p}
  \bigl(\E[G^rG^s]\bigr)_{r,s\le t}\,,
  \qquad
  G_{m,t}\xto{p}
  \bigl(\E[G^rG^s]\bigr)_{r,s\le t}\,.
\]
It follows that
\begin{equation}\label{eq:vector-coefficient-identity}
  G_{m,t}^{-1}Y_t^*K^t
  =\widetilde{a}_t + o_{\Prob}(1)\,.
\end{equation}

For the first conditional-mean term in \eqref{eq:conditional-Astar-Mt}, recall that $\ons_u^t =\sum_{r=0}^{t} a_{t,r} m^r$, so
\begin{align}
  \cA^* M^s &= u^{s+1} + \ons_u^s \notag\\
  &=u^{s+1} + Q_t a_s^{(t)}\,,
  \qquad 0\le s<t\,.\label{eq:X-columns-vector}
\end{align}
Together, \eqref{eq:vector-coefficient-identity} and \eqref{eq:X-columns-vector} yield the vector-side Onsager cancellation
\begin{align}
  X_{t-1}\gamma_t
  +Q_tG_{m,t}^{-1}Y_t^*K^t
  &=\sum_{0\le s<t}(\gamma_t)_s u^{s+1}
    +Q_t\sum_{0\le s<t}(\gamma_t)_s a_s^{(t)}
    +Q_t\widetilde{a}_t
    +Q_t o_{\Prob}(1) \notag\\
  &=\sum_{0\le s<t}(\gamma_t)_su^{s+1}
    +Q_t a_t^{(t)}+ Q_t o_{\Prob}(1) \notag\\
  &=\sum_{0\le s<t}(\gamma_t)_su^{s+1}
    +\ons_u^t+Q_t o_{\Prob}(1)\,.
  \label{eq:combined-conditional-mean-vector}
\end{align}

Subtracting $\ons_u^t$ in \eqref{eq:conditional-Astar-Mt}, and invoking Lemma \ref{lem:projected-innovation} we obtain that conditionally on $\mathcal F_t^+$,
\begin{equation}\label{eq:vector-innovation-decomposition}
  u^{t+1}
  \stackrel{d}{=}\sum_{0\le s<t}(\gamma_t)_su^{s+1}
   +\sqrt{\alpha_d \taud((K^t)^2)}\, z^{t+1}
   +e^{t+1}\,,
\end{equation}
 where the coordinates of $z^{t+1}$ are independent standard Gaussian variables, independent of $\mathcal F_t^+$.  
 The error $e^{t+1}$ is
the sum of two terms: $Q_t o_{\Prob}(1)$, coming from
\eqref{eq:combined-conditional-mean-vector}, and the negative of the
finite-rank projection
$P_{Q_t}\widetilde\cA^*K^t$.  The first term is negligible in
$\norm{\cdot}_n$ since the number of columns is fixed and the columns
of $Q_t$ have bounded normalized norms.  The second is negligible by
Lemma~\ref{lem:projected-innovation}. For the higher moments, write
\[
 P_{Q_t}\widetilde\cA^*K^t
 =
 Q_tG_{m,t}^{-1}Q_t^*\widetilde\cA^*K^t\,.
\]
Conditionally on the past, the vector
$Q_t^*\widetilde\cA^*K^t \in \R^{t+1}$ is $O_{\Prob}(n^{-1/2})$.
Since $G_{m,t}^{-1}=O_{\Prob}(1)$ and the finitely many columns of
$Q_t$ have bounded empirical moments of every order, the projection
has bounded empirical moments as well.  The same is immediate for
$Q_t o_{\Prob}(1)$.  
Hence for every $p\ge1$,
\begin{equation}\label{eq:error_vec}
 \norm{e^{t+1}}_n\xto{p}0,
 \qquad
 \frac1n\sum_{i=1}^n\abs{e_i^{t+1}}^{2p}=O_{\Prob}(1)\,.
\end{equation}

The coefficients $\gamma_t$ converge to the coefficients $\gamma_t^\infty$ of the $L^2(\mathcal{N}, \tau)$-projection of $W^t$ onto
$\Span\{W^0,\ldots,W^{t-1}\}$, and
\begin{equation}\label{eq:K-var-limit}
  \taud((K^t)^2) ~\xto{p}~
  q_{u,t}:=
  \dist_{L^2(\mathcal{N},\tau)}
  \left(W^t,\Span\{W^0,\ldots,W^{t-1}\}\right)^2\,.
\end{equation}
Therefore \eqref{eq:vector-innovation-decomposition} mirrors the ordinary Gaussian regression construction
\[
  U^{t+1}
  =\sum_{0\le s<t}(\gamma_t^\infty)_sU^{s+1}
   +\sqrt{\alpha q_{u,t}}\,Z_{t+1}\,,
\]
where $Z_{t+1}\sim N(0,1)$ is independent of $(U^0,\ldots,U^t)$.  In particular,
\[
  \E[U^{t+1}U^{r+1}]
  =\alpha\taup(W^t W^r)\,,
  \qquad 0\le r\le t\,,
\]
and $U^{t+1}$ is independent of $U^0$ jointly with the other Gaussian random variables. 
Apply Lemma~\ref{lem:conditional-empirical} to
\eqref{eq:vector-innovation-decomposition}, with the past coordinate vector
$(u_i^0,\ldots,u_i^t,m_i^0,\ldots,m_i^t)$.
The regression coefficients and variance converge by the preceding
calculations, and the error satisfies \eqref{eq:error_vec} together with the
required empirical moment bounds.  We obtain
\[
  (u_i^0,\ldots,u_i^{t+1},m_i^0,\ldots,m_i^t)_{i\le n}
  \xto{\mathrm{emp}}
  (U^0,\ldots,U^{t+1},G^0,\ldots,G^t)\,.
\]  
Now append
\[
 m_i^{t+1}=g_{t+1}(u_i^0,\ldots,u_i^{t+1})\,,
 \qquad
 G^{t+1}=g_{t+1}(U^0,\ldots,U^{t+1})\,,
\]
on each side of the empirical convergence above using Lemma~\ref{lem:emp-calc}. This closes the induction $\mathcal H_{t+1}$.
This completes the proof of Theorem~\ref{thm:direct_SE}.


\section{The terminal query}
\label{sec:adapted-terminal-query}
Here we finish the proof of Theorem~\ref{thm:main-general}. It remains to append the terminal message $\widehat m$ to the joint empirical convergence~\eqref{eq:vector-SE-main1}, and prove a stability of the difference $\cA (\widehat m - m^T)$.  

\begin{lemma}
\label{lem:ae-empirical-transform}
Suppose $(X_i^d)_{i\le n}\xrightarrow{\mathrm{emp}}X$, where $X$
has moments of all orders. Let $f:\R^k\to\R$ be bounded and Borel,
and assume
\[
 \Prob\bigl(X\in\operatorname{Disc}(f)\bigr)=0\,.
\]
Then
\[
 (X_i^d,f(X_i^d))_{i\le n}
 \xrightarrow{\mathrm{emp}}(X,f(X))\,.
\]
In particular, if $f_a$ is a bounded continuous function, then 
\[
 \frac1n\sum_{i=1}^n
 |f_a(X_i^d)-f(X_i^d)|^2
 \xrightarrow{p}
 \E|f_a(X)-f(X)|^2\,.
\]
\end{lemma}

\begin{proof}
Let $\widehat \mu$ denote the empirical measure of $(X_i^d)_{i\le n}$ and $\mu = \operatorname{Law}(X)$. Empirical convergence implies weak convergence in probability of $\widehat{\mu}$ to $\mu$ and convergence of every polynomial moment.
Letting $F(x) = (x,f(x)) \in \R^{k+1}$, we have $\operatorname{Disc}(F) = \operatorname{Disc}(f)$, therefore $\mu(\operatorname{Disc}(F))=0$. By the continuous mapping theorem applied to $F$ we have  
 $\widehat \nu := F_{\#}\widehat\mu\to \nu := F_\# \nu$.  
Now let $m\ge1$ and $\Psi\in\mathrm{PL}_m(\mathbb R^{k+1})$ and consider its cutoff 
\[\Psi_R(x,y) = \chi_R(x)\vartheta(y)\Psi(x,y)\] 
where $\chi_R \in C(\R^k,\R)$ with 
\[\chi_R(x)=1\,,~~~ \|x\|\le R\,,~~~  \chi_R(x)=0\,,~~~ \|x\|\ge 2R\,,\]
and similarly let $\vartheta \in C(\R,\R)$ be compactly supported with $\vartheta(y)=1$ for $|y|\le B := \|f\|_{\infty}$.
On the resulting compact set, $\Psi_R$ is bounded and
we have $\int \Psi_R \mathrm{d}\widehat \nu \to \int \Psi_R \mathrm{d} \nu $.
Now we remove the cutoff. 
$\Psi$ has at most polynomial growth: $|\Psi(w)|\le C_{\Psi}(1+\|w\|^m)$, $w \in \R^{k+1}$, therefore $|\Psi(x,f(x))|\le C_{\Psi,B}(1+\|x\|^m)$, $x\in\R^k$. Choose $r>0$. We have for $R \ge 1$, 
\[
\left|\int (\Psi_R - \Psi)\rmd \widehat \nu\right| \le C_{\Psi,B}\int_{\|x\|>R} (1+\|x\|^m) \rmd \widehat\mu
\le \frac{C_{\Psi,B}}{R^{r}}\int \|x\|^{m+r}\rmd \widehat\mu\,.
\]
Since $X$ has moments of all orders, the above empirical moment is $O_{\Prob}(1)$.  
Similarly,
\[\left|\int (\Psi_R - \Psi)\rmd \nu\right| \le \frac{C_{\Psi,B}}{R^{r}}\E[\|X\|^{m+r}]\,.\] 
The triangle inequality then sending $R \to \infty$ after $n\to \infty$ finishes the proof.
\end{proof}

Recall the sigma-field $\mathscr F_T$ produced by the alternating observation scheme \eqref{eq:sigma_fieldG}--\eqref{eq:sigma_fieldF}.
\begin{lemma}
\label{lem:terminal-query-continuity}
For $T$, let $z\in H_n$ be $\mathscr F_T$-measurable. Then there exists a family of random variables $C_{d,T}$ such that
\begin{align*}\label{eq:terminal-query-continuity}
\|\cA z\|_{\op}
 \le C_{d,T}\|z\|_n\,,\qquad 
 \mbox{and} \qquad \lim_{L \to \infty}\limsup_{d\to\infty}
 \Prob\bigl(C_{d,T}>  L\bigr) = 0\,.
\end{align*}
\end{lemma}

\begin{proof}
Let
\[
 Qe_s=m^s\,,\qquad Re_s=M^s\,,\qquad 0\le s<T\,,
\]
and write
\[
 G_Q=Q^*Q\,,\qquad G_R=R^*R\,,\qquad
 Y=\cA Q\,,\qquad X=\cA^*R\,.
\]
For $z$ $\mathscr F_T$-measurable, let
\[
 \beta=G_Q^{-1}Q^*z\,,
 \qquad
 r=P_Q^\perp z\,.
\]
The inverses exist almost surely by Proposition~\ref{prop:exact-rank}.
Lemma~\ref{lem:adaptive-conditioning} yields
\begin{equation}\label{eq:terminal-continuity-conditioning}
 \cA z 
 \stackrel{\mathrm d}{=}
 Y\beta+RG_R^{-1}X^*r
 +P_R^\perp\widetilde\cA r\,.
\end{equation}
State evolution and Lemma~\ref{lem:limiting-Gram-positivity} imply
\[
 \|G_Q^{-1}\|_{\op}+\|G_R^{-1}\|_{\op}=O_{\Prob}(1)\,.
\]
They also imply
\[
 \max_{s<T}
 \left\{
 \|m^s\|_n,\|M^s\|_{\op},
 \|\cA m^s\|_{\op},\|\cA^*M^s\|_n
 \right\}
 =O_{\Prob}(1)\,.
\]
 Since $T$ is fixed, Cauchy--Schwarz yields, uniformly in $z$,
\begin{equation}\label{eq:terminal-continuity-past}
 \|Y\beta\|_{\op}
 +\|RG_R^{-1}X^*r\|_{\op}
 \le C_{d,T}^{\rm past}\|z\|_n\,,
 \qquad
 C_{d,T}^{\rm past}=O_{\Prob}(1)\,.
\end{equation}

Conditionally on the history and on $r\ne0$,
$\widetilde\cA r/\|r\|_n$ is a standard GOE matrix $W_d$. Lemma~\ref{lem:projected-innovation} yields 
\[
\begin{aligned}
\|P_R^\perp\widetilde{\cA}r\|_{\op}
&\le \|\widetilde{\cA}r\|_{\op}
     +\|P_R\widetilde{\cA}r\|_{\op}\\
&\le \|r\|_n\bigl(\|W_d\|_{\op}+O_{\Prob}(d^{-1/2})\bigr),
\end{aligned}
\]
 Combining this with
\eqref{eq:terminal-continuity-conditioning} and
\eqref{eq:terminal-continuity-past} proves 
the required bound $\|\cA z\|_{\op} \le C_{d,T} \|z\|_n$ with 
\[C_{d,T} = C_{d,T}^{\rm past} + \|W_d\|_{\op} +O_{\Prob}(d^{-1/2})\,. \]
Uniform tightness follows since $C_{d,T}^{\rm past} = O_{\Prob}(1)$ and $\|W_d\|_{\op}$ has a decaying tail in $d$.
\end{proof}

As we mentioned in the introduction, the fact that $z$ is $\mathscr F_T$-measurable (i.e., obtained by at most $2T$ linear measurements of $\cA$ and $\cA^*$) is crucial to the above lemma: a worst case bound via the norm $\|\cA\|_{(\R^n,\|\cdot\|_n)\to(\Symm,\|\cdot\|_{\op})}$ is useless since the later is of order $\sqrt d$.
    
Now we characterize the limit of an extra ``vector-to-matrix half-step" of the AMP iteration. 
Run the AMP iteration~\eqref{eq:AMP0}-\eqref{eq:AMP1} up to time $T-1$. The output trajectory is $(u^0\ldots,u^T)$, $(V^0,\ldots,V^{T-1})$.
 At the next step let $k:\R^{T+1}\to\R$ be $C^1$, and suppose that
$k,\partial_0k,\ldots,\partial_Tk$ are pseudo-Lipschitz of finite
order.  Define
\[
 \widehat z_i:=k(u_i^0,\ldots,u_i^T)\,,
 \qquad
 K:=k(U^0,\ldots,U^T)\,,
\]
and
\[
 d_{k,r}:=\frac1n\sum_{i=1}^n
 \partial_rk(u_i^0,\ldots,u_i^T),
 \qquad
 c_{k,r}:=\E[\partial_rk(U^0,\ldots,U^T)],
 \quad 1\le r\le T\,.
\]
Let the extra half-step be
\begin{equation}\label{eq:extra-half-step}
 \widehat{V}:=
 \cA \widehat z - \alpha_d\sum_{r=1}^T d_{k,r}M^{r-1}\,.
\end{equation}
A consequence of the just-proved state evolution is the following joint limit. 
\begin{lemma}
\label{lem:smooth-terminal-half-step}
There is a centered semicircular variable $\widehat S$, jointly
semicircular with $S^0,\ldots,S^{T-1}$, such that
\[
 \taup(\widehat S^2)=\E[K^2]\,,
 \qquad
 \taup(\widehat S S^s)=\E[KG^s]\,,
 \quad 0\le s<T\,,
\]
and
\[
 (V^0,\ldots,V^{T-1},\widehat V,M^0,\ldots,M^{T-1})
 \xrightarrow{\mathrm{str}}
 (S^0,\ldots,S^{T-1},\widehat S, W^0,\ldots,W^{T-1})\,.
\]
Consequently,
\begin{equation}\label{eq:smooth-terminal-strong-limit}
 \cA \widehat z
 \xrightarrow{\mathrm{str}}
 \widehat S + \alpha\sum_{r=1}^T c_{k,r}W^{r-1}\,.
\end{equation}
\end{lemma}

\subsection{Proof of Theorem~\ref{thm:main-general}}
We are now ready to finish the proof of our main theorem.
Let
\[
 X=(U^0,\ldots,U^T),
 \qquad
 x_i=(u_i^0,\ldots,u_i^T),
 \qquad
 K=H-G^T.
\]
To verify joint empirical convergence, we apply
Lemma~\ref{lem:ae-empirical-transform} to the tuple 
\[\bigl(u_i^0,\ldots,u_i^{T},m_i^0,\ldots,m_i^T\bigr)_{i\le n}\] 
appearing in \eqref{eq:vector-SE-main}, with $f$ applying $h$ to the first $T$ coordinates and ignoring the rest. This is a bounded measurable function, and by Assumption~\ref{ass:zero-discontinuity}, $\operatorname{Disc}(f)$ has zero probability. Hence we can adjoin $h(x_i)$ to the empirical convergence:
\[\bigl(u_i^0,\ldots,u_i^{T},m_i^0,\ldots,m_i^T,\widehat m_i\bigr)_{i\le n}
  \xrightarrow{\mathrm{emp}}
  \bigl(U^0,\ldots,U^{T},G^0,\ldots,G^T, H\bigr)\,.\]
Now let $\mu=\Law(X)$.  We choose
$h_a\in C_c^\infty(\R^{T+1})$ such that
\begin{equation}\label{eq:approx_h}
\|h_a-h\|_{L^2(\mu)}\longrightarrow 0 \,.	
\end{equation}
Define
\[
 k_a=h_a-g_T\,,
 \qquad
 z_i^{(a)}=k_a(x_i)\,,
 \qquad
 K_a=k_a(X)\,,
 \qquad
 \widehat m_i = h(x_i)\,,\qquad 
 z=\widehat m-m^T\,.
\]
Lemma~\ref{lem:ae-empirical-transform} yields
\[
 \|z^{(a)}-z\|_n^2
 \xrightarrow{p}
 \E[(h_a(X)-h(X))^2]\,.
\]
The right-hand side tends to 0 as $a \to \infty$ by~\eqref{eq:approx_h}. Therefore Lemma~\ref{lem:terminal-query-continuity} implies, for every $\eta>0$,
\begin{equation}\label{eq:terminal-smooth-operator-error}
 \lim_{a\to\infty}\limsup_{d\to\infty}
 \Prob\bigl(\|\cA(z^{(a)}-z)\|_{\op}>\eta\bigr)=0\,.
\end{equation}
Fix $a$.  The map $k_a$ and its first derivatives are
pseudo-Lipschitz of finite order.  By
Lemma~\ref{lem:smooth-terminal-half-step} with $\widehat z = z^{(a)}$,
\begin{equation}\label{eq:strong_z_a}
 \cA z^{(a)}
 \xrightarrow{\mathrm{str}}
 L_a:=
 S_a+\alpha\sum_{r=1}^Tc_{a,r}W^{r-1},
 \qquad
 c_{a,r}:=\E[\partial_rk_a(X)]\,.
\end{equation}
We claim that 
\[
 \|L_a\| \le 2(1+\sqrt\alpha)\|K_a\|_{L^2}\,.
\]
Indeed, $\|L_a\| \le \|S_a\| + \|B_a\|$ with 
\[
 B_a:=\alpha\sum_{r=1}^Tc_{a,r}W^{r-1}\,.
\]
Since each operator is semicircular, it suffices to consider their variances.
By Lemma~\ref{lem:smooth-terminal-half-step} we have $\tau(S_a^2) = \E[K_a^2]$, and    
\[
	\tau(B_a^2) = \alpha^2 \sum_{r,r'} c_{a,r}c_{a,r'} \tau(W^{r-1}W^{r'-1})
	= \alpha \sum_{r,r'} c_{a,r}c_{a,r'} \E[U^{r}U^{r'}]\,,
\]
where the last inequality follows from the covariance recursion~\eqref{eq:covU}.
By Gaussian integration by parts, 
\begin{align*}
	\tau(B_a^2) &= \alpha \sum_{r} c_{a,r} \E[U^{r}K_a] \\
	&\le \alpha \E\Big[\big(\sum_{r} c_{a,r} U^{r}\big)^2\Big]^{1/2} \E[K_a^2]^{1/2} \\
    &= \sqrt{\alpha}\tau(B_a^2)^{1/2}\E[K_a^2]^{1/2}\,.
\end{align*}
Therefore, 
\[\tau(B_a^2) \le \alpha\E[K_a^2]\,.\]
This proves the bound on $\|L_a\|$.
The strong convergence~\eqref{eq:strong_z_a} now implies
\[
 \Prob\left(
 \|\cA z^{(a)}\|_{\op}
 >
 2(1+\sqrt\alpha)\|K_a\|_{L^2}+\eta
 \right)\longrightarrow 0\,.
\]
Finally,
\[
 \|K_a-K\|_{L^2}
 =\|h_a-h\|_{L^2(\mu)}
 \longrightarrow 0 \,.
\]
Combining this fact, \eqref{eq:terminal-smooth-operator-error}, and the
triangle inequality proves 
\[
 \Prob\left(
 \|\cA z\|_{\op}
 >
 2(1+\sqrt\alpha)\|H - G^T\|_{L^2}+\eta
 \right)\longrightarrow 0\,.
\]
With $z =\widehat m - m^T$, this is the claimed bound~\eqref{eq:terminal-comparison-opnorm}.

\paragraph{AI methodology:} The authors interacted extensively with AI to experiment with different proof strategies and adapt classical arguments to our setting.  The Gaussian conditioning technique through the induction argument of~\cite{berthier2017state} were extended with the help of GPT 5.6 Sol. The three-stage stochastic control strategy was developed with the help of GPT 6 Astra.
The ideas of designing the AMP iteration, a guess of the limiting state evolution Gaussian-semicircular process, the use of the strong convergence result~\cite{fan2021principal} in the induction argument, and the specific choices of the non-linearities designed to solve the random matrix discrepancy problem are the authors'. The paper was written by the authors.  

\paragraph{Acknowledgements:} AE was supported by the National Science Foundation grant DMS-2450867.

\bibliographystyle{alpha}

\begin{thebibliography}{KOWZ24}

\bibitem[AMS21]{ams2020}
Ahmed~El Alaoui, Andrea Montanari, and Mark Sellke.
\newblock Optimization of mean-field spin glasses.
\newblock {\em The Annals of Probability}, 49(6):2922--2960, 2021.

\bibitem[APZ19]{aubin2019storage}
Benjamin Aubin, Will Perkins, and Lenka Zdeborova.
\newblock Storage capacity in symmetric binary perceptrons.
\newblock {\em Journal of Physics A: Mathematical and Theoretical},
  52(29):294003, 2019.

\bibitem[AS22]{alaoui2022algorithmic}
Ahmed~El Alaoui and Mark Sellke.
\newblock Algorithmic pure states for the negative spherical perceptron.
\newblock {\em Journal of Statistical Physics}, 189(2):27, 2022.

\bibitem[AS26a]{akbas2026algebraic}
Emrullah Akbas and Suvrit Sra.
\newblock An algebraic matrix spencer theorem.
\newblock {\em arXiv preprint arXiv:2606.16005}, 2026.

\bibitem[AS26b]{akbas2026boolean}
Emrullah Akbas and Suvrit Sra.
\newblock Boolean small-ball inequalities for discrepancy theory.
\newblock {\em arXiv preprint arXiv:2609.20785}, 2026.

\bibitem[BB26]{bandeira2026matrix}
Afonso~S Bandeira and Helmut B{\"o}lcskei.
\newblock Matrix discrepancy for representations of finite groups.
\newblock {\em arXiv preprint arXiv:2606.12181}, 2026.

\bibitem[BJM23]{bansal2023resolving}
Nikhil Bansal, Haotian Jiang, and Raghu Meka.
\newblock Resolving matrix spencer conjecture up to poly-logarithmic rank.
\newblock In {\em Proceedings of the 55th Annual ACM Symposium on Theory of
  Computing}, pages 1814--1819, 2023.

\bibitem[BK26]{bansal2026simpler}
Nikhil Bansal and Yunbum Kook.
\newblock A simpler proof of the matrix spencer theorem.
\newblock {\em arXiv preprint arXiv:2609.16100}, 2026.

\bibitem[BMN20]{berthier2017state}
Raphael Berthier, Andrea Montanari, and Phan-Minh Nguyen.
\newblock State evolution for approximate message passing with non-separable
  functions.
\newblock {\em Information and Inference: A Journal of the IMA}, 9:33--79,
  2020.

\bibitem[Bol14]{bolthausen2014iterative}
Erwin Bolthausen.
\newblock {An iterative construction of solutions of the TAP equations for the
  Sherrington--Kirkpatrick model}.
\newblock {\em Communications in Mathematical Physics}, 325(1):333--366, 2014.

\bibitem[BS98]{biane1998stochastic}
Philippe Biane and Roland Speicher.
\newblock Stochastic calculus with respect to free brownian motion and analysis
  on wigner space.
\newblock {\em Probability Theory and Related Fields}, 112(3):373--409, 1998.

\bibitem[BS13]{bansal2013deterministic}
Nikhil Bansal and Joel Spencer.
\newblock Deterministic discrepancy minimization.
\newblock {\em Algorithmica}, 67(4):451--471, 2013.

\bibitem[DJR22]{dadush2022new}
Daniel Dadush, Haotian Jiang, and Victor Reis.
\newblock A new framework for matrix discrepancy: Partial coloring bounds via
  mirror descent.
\newblock In {\em Proceedings of the 54th Annual ACM SIGACT Symposium on Theory
  of Computing}, pages 649--658, 2022.

\bibitem[Dur19]{durrett2019probability}
Rick Durrett.
\newblock {\em {Probability: Theory and Examples}}, volume~49.
\newblock Cambridge university press, 2019.

\bibitem[FSW21]{fan2021principal}
Zhou Fan, Yi~Sun, and Zhichao Wang.
\newblock Principal components in linear mixed models with general bulk.
\newblock {\em The Annals of Statistics}, 49(3):1489--1513, 2021.

\bibitem[GKPX22]{gamarnik2022algorithms}
David Gamarnik, Eren~C K{\i}z{\i}lda{\u{g}}, Will Perkins, and Changji Xu.
\newblock Algorithms and barriers in the symmetric binary perceptron model.
\newblock {\em 2022 IEEE 63rd Annual Symposium on Foundations of Computer
  Science (FOCS)}, 2022.

\bibitem[HRS22]{hopkins2022matrix}
Samuel~B Hopkins, Prasad Raghavendra, and Abhishek Shetty.
\newblock Matrix discrepancy from quantum communication.
\newblock In {\em Proceedings of the 54th Annual ACM SIGACT Symposium on Theory
  of Computing}, pages 637--648, 2022.

\bibitem[HSS26]{huang2026algorithmic}
Brice Huang, Mark Sellke, and Nike Sun.
\newblock Algorithmic threshold for high-dimensional projection pursuit i:
  general theory.
\newblock {\em arXiv preprint arXiv:2608.29416}, 2026.

\bibitem[HT05]{haagerup2005new}
Uffe Haagerup and Steen Thorbj{\o}rnsen.
\newblock A new application of random matrices: is not a group.
\newblock {\em Annals of Mathematics}, 162:711--775, 2005.

\bibitem[JM13]{javanmard2013state}
Adel Javanmard and Andrea Montanari.
\newblock State evolution for general approximate message passing algorithms,
  with applications to spatial coupling.
\newblock {\em Information and Inference: A Journal of the IMA}, 2(2):115--144,
  2013.

\bibitem[JM14]{javanmard2013hypothesis}
Adel Javanmard and Andrea Montanari.
\newblock Hypothesis testing in high-dimensional regression under the gaussian
  random design model: Asymptotic theory.
\newblock {\em IEEE Transactions on Information Theory}, 60(10):6522--6554,
  2014.

\bibitem[Kat26]{kathuria2026walk}
Tarun Kathuria.
\newblock A walk from free probability to matrix discrepancy i: Matrix spencer.
\newblock {\em arXiv preprint arXiv:2609.18914}, 2026.

\bibitem[KL26]{kizildaug2026stable}
Eren~C K{\i}z{\i}lda{\u{g}} and Shuangping Li.
\newblock Stable and online algorithms for random matrix discrepancy.
\newblock {\em arXiv preprint arXiv:2610.01591}, 2026.

\bibitem[KOWZ24]{kunisky2024asymptotic}
Dmitriy Kunisky, Timm Oertel, Nicola Wengiel, and Peiyuan Zhang.
\newblock Asymptotic bounds and online algorithms for average-case matrix
  discrepancy.
\newblock {\em arXiv preprint arXiv:2410.23915}, 2024.

\bibitem[KZ23]{kunisky2023average}
Dmitriy Kunisky and Peiyuan Zhang.
\newblock Average-case matrix discrepancy: Asymptotics and online algorithms.
\newblock {\em arXiv preprint arXiv:2307.10055}, 2023.

\bibitem[LG16]{le2016brownian}
Jean-Fran{\c{c}}ois Le~Gall.
\newblock {\em Brownian motion, martingales, and stochastic calculus}, volume
  274.
\newblock Springer, 2016.

\bibitem[LM15]{lovett2015constructive}
Shachar Lovett and Raghu Meka.
\newblock Constructive discrepancy minimization by walking on the edges.
\newblock {\em SIAM Journal on Computing}, 44(5):1573--1582, 2015.

\bibitem[LRR17]{levy2017deterministic}
Avi Levy, Harishchandra Ramadas, and Thomas Rothvoss.
\newblock Deterministic discrepancy minimization via the multiplicative weight
  update method.
\newblock In {\em International Conference on Integer Programming and
  Combinatorial Optimization}, pages 380--391. Springer, 2017.

\bibitem[Mai25]{maillard2025average}
Antoine Maillard.
\newblock Average-case matrix discrepancy: Satisfiability bounds.
\newblock {\em Random Structures \& Algorithms}, 67(3):e70033, October 2025.

\bibitem[Mal12]{male2012norm}
Camille Male.
\newblock The norm of polynomials in large random and deterministic matrices.
\newblock {\em Probability Theory and Related Fields}, 154(3):477--532, 2012.

\bibitem[Mek14]{meka2014blog}
Raghu Meka.
\newblock Discrepancy and beating the union bound, 2014.

\bibitem[Mil16]{miller2016stochastic}
Jason Miller.
\newblock Stochastic calculus.
\newblock University of Cambridge, Part III lecture notes, Lent term, 2016.

\bibitem[Mon21]{montanari2021optimization}
Andrea Montanari.
\newblock Optimization of the sherrington--kirkpatrick hamiltonian.
\newblock {\em SIAM Journal on Computing}, (0):FOCS19--1, 2021.

\bibitem[MS17]{mingo2017free}
James~A Mingo and Roland Speicher.
\newblock {\em Free probability and random matrices}, volume~35.
\newblock Springer, 2017.

\bibitem[MZZ24]{montanari2021tractability}
Andrea Montanari, Yiqiao Zhong, and Kangjie Zhou.
\newblock Tractability from overparametrization: The example of the negative
  perceptron.
\newblock {\em Probability Theory and Related Fields}, 188:805--910, 2024.

\bibitem[{Ope}26]{OAI.fullrsb_sk}
{OpenAI}.
\newblock Full support of the zero-temperature {Sherrington--Kirkpatrick} order
  parameter.
\newblock Released on September 27, 2026. Available in the
  \href{https://github.com/openai/math/blob/main/preprints/Full-support-of-the-zero-temperature-Sherrington-Kirkpatrick-order-parameter-September-27-2026/main.pdf}{\texttt{openai/math}}
  GitHub repository, 2026.

\bibitem[Sch05]{schultz2005non}
Hanne Schultz.
\newblock Non-commutative polynomials of independent gaussian random matrices.
  the real and symplectic cases.
\newblock {\em Probability Theory and Related Fields}, 131(2):261--309, 2005.

\bibitem[Spe85]{spencer1985six}
Joel Spencer.
\newblock Six standard deviations suffice.
\newblock {\em Transactions of the American Mathematical Society},
  289(2):679--706, 1985.

\bibitem[Sub21]{subag2018following}
Eliran Subag.
\newblock Following the ground states of full-rsb spherical spin glasses.
\newblock {\em Communications on Pure and Applied Mathematics},
  74(5):1021--1044, 2021.

\bibitem[SZ26]{song2026matrix}
Zhao Song and Lichen Zhang.
\newblock Matrix spencer: Eight standard deviations suffice and an
  almost-linear time algorithm for dense input.
\newblock {\em arXiv preprint arXiv:2609.15025}, 2026.

\bibitem[vH25]{van2025strong}
Ramon van Handel.
\newblock The strong convergence phenomenon.
\newblock {\em arXiv preprint arXiv:2507.00346}, 2025.

\bibitem[Voi91]{voiculescu1991limit}
Dan Voiculescu.
\newblock Limit laws for random matrices and free products.
\newblock {\em Inventiones Mathematicae}, 104(1):201--220, 1991.

\bibitem[Zou12]{zouzias2012matrix}
Anastasios Zouzias.
\newblock A matrix hyperbolic cosine algorithm and applications.
\newblock In {\em International Colloquium on Automata, Languages, and
  Programming}, pages 846--858. Springer, 2012.

\end{thebibliography}

\end{document}